\documentclass[11pt]{amsart}

\usepackage{amssymb,bbm,mathtools}
\usepackage{a4wide}

\allowdisplaybreaks[3]

\theoremstyle{plain}
\newtheorem{theorem}{Theorem}
\newtheorem{prop}{Proposition}
\newtheorem{lemma}{Lemma}
\newtheorem{coro}{Corollary}

\theoremstyle{definition}
\newtheorem{remark}{Remark}

\newcommand{\dd}{\,\mathrm{d}}
\newcommand{\ii}{\ts\mathrm{i}\ts}
\newcommand{\ts}{\hspace{0.5pt}}

\newcommand{\Gam}{\varGamma}
\newcommand{\Lam}{\varLambda}

\newcommand{\ZZ}{\mathbb{Z}}
\newcommand{\RR}{\mathbb{R}\ts}
\newcommand{\CC}{\mathbb{C}\ts}
\newcommand{\QQ}{\mathbb{Q}\ts}
\newcommand{\NN}{\mathbb{N}}

\def\sbe{\subseteq}

\newcommand{\oo}{{\scriptstyle \mathcal{O}}}
\newcommand{\OO}{\mathcal{O}}
\newcommand{\CR}{\mathcal{R}}

\newcommand{\RE}{\mathop{\mathrm{Re}}\nolimits}
\newcommand{\IM}{\mathop{\mathrm{Im}}\nolimits}
\newcommand{\tr}{\mathop{\mathrm{tr}}\nolimits}
\newcommand{\card}{\mathop{\mathrm{card}}\nolimits}

\newcommand{\Fix}{\mathop{\mathrm{Fix}}\nolimits}
\newcommand{\OC}{\mathop{\mathrm{OC}}\nolimits}

\newcommand{\id}{\mathop{\mathrm{id}}\nolimits}
\newcommand{\lcm}{\mathop{\mathrm{lcm}}\nolimits}

\newcommand{\myfrac}[2]{\frac{\raisebox{-2pt}{$#1$}}{\raisebox{0.5pt}{$#2$}}}
  
\newcommand{\rect}{\raisebox{0.7pt}{$\scriptstyle \sqsubset \!\!
\sqsupset$}}

\newcommand{\nn}{\nonumber}

\begin{document}

\title[Well-rounded sublattices]{Well-rounded sublattices of planar lattices}

\author{Michael Baake}
\address{Fakult\"at f\"ur Mathematik, Universit\"at Bielefeld, 
Box 100131, 33501 Bielefeld, Germany}
\email{$\{$mbaake,pzeiner$\}$@math.uni-bielefeld.de}
% \urladdr{http://www.math.uni-bielefeld.de/baake}

\author{Rudolf Scharlau}
\address{Fakult\"at f\"ur Mathematik, 
Technische Universit\"at Dortmund, 44221 Dortmund, Germany}
\email{Rudolf.Scharlau@math.tu-dortmund.de}
% \urladdr{http://www.mathematik.uni-dortmund.de/~scharlau/}

\author{Peter Zeiner}

% \date{\today}

\begin{abstract}
  A lattice in Euclidean $d$-space is called well-rounded if it
  contains $d$ linearly independent vectors of minimal length. This
  class of lattices is important for various questions, including
  sphere packing or homology computations. The task of enumerating
  well-rounded sublattices of a given lattice is of interest already
  in dimension 2, and has recently been treated by several authors.
  In this paper, we analyse the question more closely in the spirit of
  earlier work on similar sublattices and coincidence site
  sublattices. Combining explicit geometric considerations with known
  techniques from the theory of Dirichlet series, we arrive, after a
  considerable amount of computation, at asymptotic results on the
  number of well-rounded sublattices up to a given index in any planar
  lattice.  For the two most symmetric lattices, the square and the
  hexagonal lattice, we present detailed results.
\end{abstract}
  
\maketitle

\section{Introduction}

A lattice in Euclidean space $\RR^d$ is \emph{well-rounded} if the
non-zero lattice vectors of minimal length span $\RR^d$.  Well-rounded
lattices are interesting for several reasons.  First of all, the
concept is put into a broader context by the notion of the
\emph{successive minima} of a lattice (more precisely, of a norm
function on a lattice). By definition, a lattice is well-rounded if
and only if all its $d$ successive minima (norms of successively
shortest linearly independent vectors) are equal to each other.

A first observation is that many important `named' lattices in
higher-dimensional space are well-rounded, such as the Leech lattice,
the Barnes-Wall lattice(s), the Coxeter-Todd lattice, all irreducible
root lattices, and many more~\cite{ConSlo}.  There are essentially two
reasons for this (which often apply both). First of all, distinct
successive minima give rise to proper subspaces of $\RR^d$ that are
invariant under the orthogonal group (automorphism group) of the
lattice. If this finite group acts irreducibly on $\RR^d$, the lattice
must be well-rounded. Secondly, a lattice which gives rise to a
locally densest sphere packing (a so-called extreme lattice), is
well-rounded. It is actually perfect by Voronoi's famous theorem (this
part goes back to Korkine and Zolotareff), and it is easily seen that
perfection implies well-roundedness; compare \cite{Martinet}.

However, these two observations are not at the core of the
notion. They might give the impression that well-rounded lattices are
very rare or special, which is not the case. In terms of Gram matrices
or quadratic forms, the well-rounded ones lie in a subspace of
codimension $d-1$ in the space of all symmetric matrices, similarly
for the cone of positive definite Minkowski-reduced forms. Despite its
codimension, this subspace is large enough so that certain questions
about general forms can be reduced to well-rounded ones. A good
illustration for this is Minkowski's proof of the fact that the
geometric mean of all $d$ successive minima of a lattice is bounded by
the same quantity $\gamma^{}_d \cdot \mathrm{disc}(\Lam)$ as the
first minimum (see Section~\ref{sec:tools}).
Here, $\gamma^{}_d$ is the Hermite constant in
dimension $d$, and for well-rounded lattices this estimate reduces to
the definition of this constant. The proof is obtained by a certain
deformation of the quadratic form; see \cite{vanderWaerden}. A
sharpened version of this technique asks for a diagonal matrix which
transforms a given lattice into a well-rounded one. In general, its
existence is unknown, but C.~McMullen~\cite{McMullen}
recently proved a weaker version
which suffices for applications to Minkowski's conjecture on the
minimum of a (multiplicative) norm function on lattices. The method of
proof is related to applications of well-rounded lattices to
cohomology questions as described in the introduction of~\cite{Kuehn};
compare the references given there.

Having this kind of `richness' of well-rounded lattices in mind, it is
tempting to ask how frequent they are in terms of counting
sublattices. So, the principal object of study in this paper is the
function
\begin{equation}\label{eq:def-of-a}
  a^{}_\Gam (n)\, :=\, \card \{ \Lam \mid \Lam \subseteq
  \Gam \text{ is a well-rounded sublattice with } [\Gam : \Lam]
  = n\},
\end{equation}
where $\Gam$ is an in principle arbitrary lattice, and $[\Gam :
\Lam]$ denotes the index of $\Lam$ in $\Gam$.  This
question is of interest already in dimension $2$ (where some of the
general features described above reduce to rather obvious
facts). Moreover, since the well-rounded sublattices are the objects
of interest, and not so much the enveloping `lattice of reference'
$\Gam$, it seems natural to focus mainly on the two most symmetric
lattices, the hexagonal lattice and the square lattice. In this
paper, we shall obtain complete and explicit results on the asymptotic
number of well-rounded sublattices, as a function of the index, of the
hexagonal lattice and of the square lattice. We also have results for
general $\Gam$ which are somewhat weaker, which seems to be
unavoidable.\medskip

In special situations, lattice enumeration problems have a long
history. The coefficients of the Dedekind zeta functions of an
algebraic number field $K$ of degree $d$ over the rationals count the
number of ideals of given index in the ring of integers $\ZZ_K$, which
is considered as a lattice in a well-known way~\cite{Borevic-Sh}.  The
perhaps most basic result on lattice enumeration, which is also one of
the most frequently rediscovered ones, is the determination of the
number $g(n)$ of all distinct sublattices of index $n$ in a given
lattice $\Gam\subset \RR^d$.  The result follows easily from the
Hermite normal form for integral matrices and reads
\begin{equation}\label{eq:Zn-count}
  g^{}_d(n) \, =\, g(n) \, =\,
   \sum_{m_1\cdot\ldots\cdot m_d=n} m_1^0\cdot m_2^1\cdots m_d^{d-1}
\end{equation}
with Dirichlet series generating function 
\begin{equation}\label{eq:Zn-Diric}
   D^{}_g(s) \, = \, \sum_{n=1}^\infty \frac{g(n)}{n^s}
  \, = \, \zeta(s)\zeta(s-1)\cdots\zeta(s-d+1)
\end{equation}
(compare \cite[p.\ 64]{Shimura}, \cite[p.\ 307]{Sol}, \cite{Lub,Baake97};
for several different proofs, see \cite[Theorem~15.1]{Lub}).  This
result of Eq.~\eqref{eq:Zn-count} is insensitive to any geometric
property of the lattice $\Gam$, in the sense that it is actually a
result for the free Abelian group of rank $d$ and its subgroups. In
\cite{duSautoy, Grune}, extensions to more general classes of finitely
generated groups are treated.

As for lattices, it is natural to refine the question by looking at
classes of sublattices with particular properties (number-theoretic or
geometric), possibly defined by an additional structure on the
enveloping vector space. In addition to the classical case of the
Dedekind zeta function mentioned above, we are aware of only few,
scattered results. Quite a while ago, in \cite{Sol, Bush}, modules in
an order in a semisimple algebra over a number field were
considered. Well-rounded lattices in dimension $2$ have recently been
analysed in \cite{Fuksh1, Fuksh3, Fuksh2, Kuehn}; see also the
references in \cite{Fuksh2}. Together with our earlier work on similar
sublattices \cite{Baake+Grimm,BSZ11} and on coincidence site
sublattices (CSLs) \cite{Baake97,pzcsl2,BGHZ08,zou}, these papers were
our starting point.

One benefit of Dirichlet series is the access to asymptotic results on
the growth of a (non-negative) arithmetical function $f(n)$.  Since
$f$ in general need not behave regularly, in particular need not be
monotone, one usually considers the average growth of $f(n)$, that is,
one studies the summatory function $F(x)=\sum_{n\leq x} f(n)$.  For
the above counting function $g^{}_d(n)$ for sublattices, the summatory
function $G^{}_d(x)$ satisfies
\begin{equation}\label{eq:Zn-asymp}
    G^{}_d(x) \, = \, c x^d + \Delta_d(x) \ts ,
\end{equation}
with $c=1$ for $d=1$ and $c = \frac{1}{d}\prod_{\ell=2}^d \zeta(\ell)$
otherwise, which follows from Eq.~\eqref{eq:Zn-Diric} by applying
Delange's theorem; compare Theorem~\ref{thm:meanvalues} in
Appendix~\ref{app-sec:ana}.  Clearly, $G^{}_1(x) = [x] $, where
$[\cdot]$ denotes the Gauss bracket,
and thus
$\Delta_1(x) = \OO(1)$.  In dimension 2, $G^{}_2 = \sigma_1(n) :=
\sum_{\ell \mid n}\ell$, so we have the well-known asymptotic growth
behaviour of the divisor function, whose error term can be estimated
as \mbox{$\Delta_2(x) = \OO\bigl(x\log (x) \bigr)$};
see~\cite[Thm~3.4]{Apostol}.

One can ask for a more refined description of the asymptotic growth of
an arithmetic function, consisting of a main term for the summatory
function, a term of second order (a `first order error term'), and an
error term of a strictly smaller order of magnitude than the term of
second order.  For instance, for the number of divisors of $n$, it is
known that
\begin{equation}\label{eq:div-asympt}
   \sum_{n\leq x} \sigma^{}_0(n) \, = \, 
   x \log(x) + (2\gamma - 1)\ts x + \OO\bigl(\sqrt{x}\ts \bigr) \ts ,
\end{equation}
where $\gamma$ is the Euler--Mascheroni constant;
compare~\cite{Apostol,Tenenbaum}. So we have a term of second order which
is linear in this case and thus of `almost the same' growth as the
main term, whereas the error term is much smaller. \medskip 

The content of this paper can now be summarised as
follows. In the short preparatory Section~\ref{sec:tools}, we recall a
few facts about reduced bases and Bravais classes of lattices in the
plane, and state some auxiliary remarks about well-rounded
(sub-)lattices.

In Section~\ref{sec:sq}, we begin with an explicit description of all
well-rounded sublattices of the square lattice, the latter viewed as
the ring $\ZZ[\ii]$ of Gaussian integers. After these preparations,
the main result then is Theorem~\ref{thm:sq2}, which gives a refined
asymptotic description of the function $A_\square$, of the kind that
we have explained above for the divisor function in
Eq.~\eqref{eq:div-asympt}; the constants for the main term and the
term of second order are determined explicitly.  The proof relies on
classic methods from analytic number theory, including Delange's
theorem and some elementary tools around Euler's summation formula and
Dirichlet's hyperbola method.  We describe the strategy and the main
steps of the proof; some of the details, which are long and technical,
have been transferred to a supplement to this paper.  A weaker result,
namely the explicit asymptotics without the second-order term, is
stated in Theorem~\ref{thm:sq1}, which is fully proved here.

Section~\ref{sec:tr} provides the analogous analysis for the hexagonal
lattice, realised as the ring of Eisenstein integers $\ZZ[\rho]$ with
$\rho=e^{2\pi \ii /3}$; Theorems~\ref{thm:tri1} and~\ref{thm:tri2} are
completely analogous to Theorems~\ref{thm:sq1} and~\ref{thm:sq2}.
\medskip

The general case of well-rounded sublattices of two-dimensional case
is treated in Section~\ref{sec:gen}, which is subdivided into two
parts.  The first one starts with a criterion for the existence of
well-rounded sublattices. The lattices that have a well-rounded
sublattice include all `rational' lattices, that is, lattices whose
Gram matrix consists of rational numbers (or even rational integers),
up to a common multiple. So these are exactly the lattices that
correspond to integral quadratic forms in the classical sense.  There
is an interesting connection between well-rounded sublattices and
CSLs, which is established in Lemma~\ref{lem:BRS}.  In the rest of
this part, it is shown in Theorem~\ref{thm:non-rat-growth} that all
non-rational lattices that contain well-rounded sublattices have
essentially the same power-law growth (linear) of their average number
$A_\Gam (x)$. The second part of Section~\ref{sec:gen} deals with the
behaviour of $A_\Gam (x)$ in the general rational case. The discussion
is more complicated, but nevertheless we can show that the growth rate
is proportional to $x \log (x)$, as in the square and hexagonal case.
Summarising, we see that three regimes exist as follows: A planar
lattice can have many, some or no well-rounded sublattices, the first
case is exactly the rational case, while the second case is explained
by the existence of an essentially unique coincidence
reflection.\medskip

Our paper is complemented by four appendices. In
Appendix~\ref{app-sec:ana}, some classic results about Dirichlet
series are collected in a way that suits our needs. In
Appendix~\ref{app-sec:asym}, we explicitly record the asymptotic
behaviour of the number of similar sublattices of the square and the
hexagonal lattice, which are a useful by-product of Sections~\ref{sec:sq}
and~\ref{sec:tr}. Appendix~\ref{app-sec:index} summarises
key properties of a special type of sublattices that we need,
while Appendix~\ref{app-sec:epst} recalls some facts about Epstein's
zeta functions.

\section{Tools from the geometry of planar lattices}\label{sec:tools}

Let us collect some simple, but useful facts from the geometric theory
of lattices.  We assume throughout this paper that we are in dimension
$d=2$, so we consider an arbitrary lattice $\Lam$ in the Euclidean
plane. Let $v \in \Lam$ be a shortest non-zero vector, and $w \in
\Lam$ shortest among the lattice vectors linearly independent from
$v$. Then $v, w$ form a basis of $\Lam$.  (The reader may consult
\cite[Chapter 2, \S 7.7]{Borevic-Sh} for this and for related statements
below.)  Changing the sign of $w$ if necessary, we may assume that the
inner product satisfies $(v,w ) \ge 0$. A basis of this kind is called
a \emph{reduced basis} of $\Lam$.  By definition, we have the
following chain of inequalities,
\begin{equation}\label{eq:red-basis}
  \left|v\right| \, \le \, \left|w\right|
  \, \le \, \left|v-w\right| \, \le \, \left|v+w\right| \ts .
\end{equation}
In terms of the quantities $a:= |v|^2, c:=|w|^2$, and $b:= ( v,w)$,
which are the entries of the Gram matrix $\left(\begin{smallmatrix} a
    & b \\ b & c \end{smallmatrix}\right)$ with respect to $v, w$,
these conditions read
\begin{equation}\label{eq:red-gram}
  0 \, \le \, 2b \, \le \, a \, \le \, c \ts .
\end{equation} 
Conversely, if we start with any two linearly independent vectors
$v,w$ satisfying Eqs.~(\ref{eq:red-basis}) or (\ref{eq:red-gram}),
then $v,w$ form a reduced basis of the lattice that they generate.
Concerning the reduction conditions (\ref{eq:red-basis}), there are
six cases possible for the pair $v, w$ as follows,
\begin{align*}
  \text{(a)}\qquad
  & \left|v\right| < \left|w\right| < \left|v-w\right| <
  \left|v+w\right|, \quad ( v,w ) > 0 & \qquad 
  \text{general type} \\
  \text{(b)}\qquad
  & \left|v\right| < \left|w\right| < \left|v-w\right| =
  \left|v+w\right|, \quad ( v,w ) = 0 & \qquad 
  \text{rectangular type} \\
  \text{(c)}\qquad
  & \left|v\right| < \left|w\right| = \left|v-w\right| <
  \left|v+w\right|, \quad ( v,w ) > 0 & \qquad
  \text{centred rectangular type} \\
  \text{(d)}\qquad
  & \left|v\right| = \left|w\right| < \left|v-w\right| <
  \left|v+w\right|, \quad ( v,w ) > 0 & \qquad 
  \text{rhombic type} \\
  \text{(e)}\qquad
  & \left|v\right| = \left|w\right| < \left|v-w\right| =
  \left|v+w\right|, \quad ( v,w ) = 0 & \qquad 
  \text{square type} \\
  \text{(f)}\qquad
  & \left|v\right| = \left|w\right| = \left|v-w\right| <
  \left|v+w\right|, \quad ( v,w ) > 0 & \qquad 
  \text{hexagonal type} 
\end{align*}
It is well-known and easily shown that the entries $a,b,c$ of the Gram
matrix with respect to a reduced basis $v, w$, only depend on the
lattice, but not on the choice of the reduced basis $v,w$.  Therefore,
it is well-defined to talk about the \emph{geometric type} of the
lattice, which is one of the types (a) to (f) above.  As a further
consequence of this uniqueness property, the orthogonal group $\mathrm
O (\Lam)$ acts transitively (and thus sharply transitively) on the set
of all (ordered) reduced bases of $\Lam$. (By definition, $\mathrm O
(\Lam)$ is the set of orthogonal transformations of the enveloping
vector space which maps the lattice into, and thus onto itself.)
$\mathrm O (\Lam)$ is cyclic of order 2 for lattices of general type,
a dihedral group of order $4$ (generated by two perpendicular
reflections) for the types (b), (c) and (d), a dihedral group of order
$8$ for the square lattice, and of order $12$ for the hexagonal
lattice.

Typically, one wants to classify lattices only up to similarity, which
means that the Gram matrix may be multiplied with a positive
constant. Clearly, a square or hexagonal lattice is unique up to
similarity. Similarity classes of rhombic type depend on one
parameter, the angle $\alpha$ formed by $v$ and $w$, where $\pi/3 <
\alpha < \pi /2$.  The limiting cases $\alpha = \pi / 3$ and $\alpha =
\pi / 2$ lead to the hexagonal, respectively square lattice. \medskip

A lattice $\Lam$ (in any dimension) is called \emph{rational} if its
similarity class contains a lattice with rational Gram matrix.  The
\emph{discriminant} $\mathrm{disc}(\Lam)$
of a lattice $\Lam$ is the determinant of any of
its Gram matrices. (This is the square of the volume of a fundamental
domain for the action of $\Lam$ by translations.)

Two lattices $\Gam, \Lam$ (on the same space) are called
\emph{commensurate} (or commensurable) if their intersection $\Gam
\cap \Lam$ has finite index in both. Equivalently, there exists a
non-zero integer $a$ such that $a \Gam \subseteq \Lam \subseteq a^{-1}
\Gam$. This in turn is equivalent to the condition that $\Gam $ and
$\Lam$ generate the same space over the rationals, $\QQ \Gam = \QQ
\Lam$. If $\Gam$ and $\Lam$ are commensurate, the ratio of their
discriminants is a rational square.

A \emph{coincidence isometry} for $\Lam$ is an isometry (an orthogonal
transformation $R$ of the underlying real space) such that $\Lam$ and
$R \Lam$ are commensurate. In earlier work~\cite{Baake97}, we have
introduced the notation $\OC (\Lam)$ for the set of all coincidence
isometries for $\Lam$. If $R \in \OC (\Lam)$, it follows that $R\QQ
\Lam = \QQ R \Lam = \QQ \Lam$ (see above), i.e.\ $R$ induces an
orthogonal transformation of the rational space $\QQ
\Lam$. Conversely, any such orthogonal transformation maps $\Lam$ onto
a lattice of full rank in the same rational space, which, by the above
remarks, is commensurate with $\Lam$. Altogether, $\OC (\Lam)$ is
equal to the rational orthogonal group $\mathrm O (\QQ \Lam)$ (in
particular, it is a group). If $\Gam$ and $\Lam$ are commensurate,
their groups of coincidence isometries coincide, 
\[
   \OC (\Gam) \, = \, \mathrm{O} (\QQ \Gam)
   \, = \, \mathrm{O} (\QQ \Lam) \, = \, \OC (\Lam)\ts .
\]  
A \emph{coincidence site lattice} (CSL) for $\Lam$ is a sublattice of
the form $\Lam \cap R\Lam$ with $R \in \OC (\Lam)$; see \cite{Baake97}
for further motivation concerning this notion.

Geometric types as introduced above are closely related, but not
identical, with the so-called \emph{Bravais types} of lattices, which
are defined in any dimension.  Two lattices $\Gam$ and $\Lam$ are
Bravais equivalent if and only if there exists a linear transformation
which maps $\Gam$ onto $\Lam$ and also conjugates $\mathrm O (\Gam)$
into $\mathrm O (\Lam)$.  The Bravais type (or Bravais class) of a
lattice depends only on its geometric type; the centred rectangular
and the rhombic lattices belong to the same Bravais type (thus we call
them \mbox{rhombic-cr} lattices).  Otherwise, geometric types and
Bravais types (or rather the respective equivalence classes of
lattices) coincide.\medskip

Let us return to well-rounded lattices.  Clearly, a planar lattice is
well-rounded if and only if it is of rhombic, square or hexagonal
type.  Any \mbox{rhombic-cr} lattice contains a rectangular sublattice
of index $2$.  In fact, if $v$ and $w$ form a reduced basis, then
$v-w$ and $v+w$ are orthogonal, and form a reduced basis of the
desired sublattice. Conversely, if $v,w$ is a reduced basis of a
rectangular lattice, and if we further assume that $|w^2|=c < 3a =
3|v|^2$, then $v+w$ and $-v+w$ form a reduced basis of a rhombic
sublattice of index $2$. (If $c=3a$, this sublattice is hexagonal,
whereas for $c>3a$, we have $|2v|< |\pm v + w|$, and thus the vectors
are not shortest any more; in this case, the sublattice is centred
rectangular.)

Similarly, a hexagonal lattice contains a rectangular sublattice of
index $2$, or more precisely, it contains exactly three rectangular
sublattices of index $2$ for symmetry reasons.  Analogously, the
square lattice contains precisely one square sublattice of index $2$.

\section{Well-rounded sublattices of $\ZZ[\ii ]$}\label{sec:sq}

We use the Gaussian integers as a representation of the square
lattice.  Note that there is no hexagonal sublattice of $\ZZ[\ii ]$
(consider the discriminant).  Hence, all well-rounded sublattices are
either rhombic or square lattices, which we treat separately, in line
with the geometric classification explained above.

A fundamental quantity that will appear frequently below is the
Dirichlet series generating function for the number of similar
sublattices of $\ZZ[\ii]$, compare \cite{Baake+Grimm,BSZ11}, which is
equal to the Dedekind zeta function of the quadratic field $\QQ(\ii)$,
\begin{equation}\label{eq:Dedekind-square}
   \Phi^{}_{\square} (s) \, = \, \zeta^{}_{\QQ (\ii)} (s)
   \, = \, \zeta (s) \ts L (s,\chi^{}_{-4}) \, \ts .
\end{equation}
Here, $\zeta (s)$ is Riemann's zeta function, and $L(s,\chi^{}_{-4})$
is the $L$-series corresponding to the Dirichlet character
$\chi^{}_{-4}$ defined by
\[
    \chi^{}_{-4} (n) \, = \,
    \begin{cases} 0 , & \text{if $n$ is even, } \\
    1 , & \text{if $n\equiv 1 \bmod 4$, } \\
    -1 , & \text{if $n\equiv 3 \bmod 4$;}
    \end{cases}
\] 
see \cite{Baake97, BSZ11, Zagier} and Appendix~\ref{app-sec:ana}.

Before dealing with the well-rounded sublattices, let us consider all
\mbox{rhombic-cr} and square sublattices of $\ZZ[\ii ]$ (recall that
the term `\mbox{rhombic-cr}' means rhombic or centred rectangular).
Let $z_1, z_2 \in \ZZ[\ii ]$ be any two elements of equal norm. The
sublattice $\Gam=\langle z_1,z_2\rangle_\ZZ$ is of rhombic or centred
rectangular or square type, and every \mbox{rhombic-cr} or square
sublattice is obtained in this way (see Section 2).  We can write
$z_1+z_2$ and $z_1-z_2$ as $z_1+z_2=pz$ and $z_1-z_2=\ii q z$ where
$p,q$ are integers and $z$ is primitive, which means that $\RE(z)$ and
$\IM(z)$ are relatively prime. W.l.o.g., we may assume that $p$ and
$q$ are positive (interchange $z_1$ and $z_2$ if necessary).  Thus
$\Gam=\langle z_1,z_2\rangle_\ZZ= \langle \frac{p+\ii q}{2}z,
\frac{p-\ii q}{2}z\rangle_\ZZ$ is a sublattice of $\ZZ[\ii]$ of index
$\frac{1}{2}pq|z|^2$.  The lattice $\Gam$ is a square lattice if and
only if $p=q$.  Determining the number of \mbox{rhombic-cr} and square
sublattices is thus equivalent to finding all rectangular and square
sublattices of $\ZZ[\ii ]$ with the additional constraint that
$(p+q\ii)z$ is divisible by $2$.

We distinguish two cases (note that $z$ is primitive, hence, in particular,
not divisible by $2$, and thus
$p$ and $q$ must have the same parity), which we call `rectangular'
and `rhombic case' for reasons that will become clear later.
\begin{enumerate}
\item \label{cases:`rect'} `rectangular' case: $z$ is not divisible by
  $1+\ii$, hence $p$ and $q$ must be even.  We write $p=2p',
  q=2q'$. The index is even since it is given by $2p'q'|z|^2$.  Note
  that $p',q'$ may take any positive integral value, even or odd.
\item \label{cases:`rhomb'}`rhombic' case: $z$ is divisible by
  $1+\ii$. We write $z=(1+\ii)w$.
    \begin{enumerate}
    \item \label{cases:`rhomb':even} If $p$ and $q$ are both even, we
      again write $p=2p', q=2q'$.  The index is divisible by 4 since
      it is given by $4p'q'|w|^2$.  Note that $p',q'$ may take any
      positive integral value, even or odd.
    \item \label{cases:`rhomb':odd} If $p$ and $q$ are both odd, the
      index is odd and given by $pq|w|^2$.
    \end{enumerate}
\end{enumerate}
For fixed $z$, interchanging $p\neq q$ gives a \mbox{rhombic-cr} (and
rectangular) lattice which is rotated through an angle
$\frac{\pi}{2}$, hence we count no lattice twice if we let $p,q$ run
over all positive integers.

Let $\Phi_{\text{\rm even}}(s)$ be the Dirichlet series for the number
of \mbox{rhombic-cr} and square sublattices of even index. This comprises the
cases~(\ref{cases:`rect'}) and (\ref{cases:`rhomb':even}).  As $p',q'$
run over all positive integers, they each contribute a factor of
$\zeta(s)$, and since $z$ is primitive, this gives the factor
$\Phi^{\mathsf{pr}}_{\square}(s)$, where
$\Phi^{\mathsf{pr}}_{\square}(s)$ is the Dirichlet series generating
function of primitive similar sublattices of $\ZZ[\ii]$.
The additional factor of $2$
in the index formula gives a contribution of $2^{-s}$, and combining
all these factors finally yields
\begin{equation}
   \Phi_{\text{\rm even}}(s) \, = \,
   \frac{1}{2^s}\, \zeta(s)^2
   \, \Phi^{\mathsf{pr}}_{\square}(s).
\end{equation}
It remains to calculate the number of \mbox{rhombic-cr} and square
sublattices of odd
index, with generating function $\Phi_{\text{\rm odd}}(s)$. Here, $p$ and $q$ run
over all odd positive integers and hence each contribute a factor of
$(1-2^{-s})\zeta(s)$, whereas $w$ runs over all primitive $w$ with
$|w|^2$ odd, and hence gives the contribution
$\frac{1}{1+2^{-s}}\Phi^{\mathsf{pr}}_{\square}(s)$, so that we have
\begin{equation}
   \Phi_{\text{\rm odd}}(s) \, = \,
   \frac{(1-2^{-s})^2}{1+2^{-s}}
   \, \zeta(s)^2 \, \Phi^{\mathsf{pr}}_{\square}(s) \ts .
\end{equation}
In total, the generating function $\Phi_{\lozenge+\square}(s)$ for the number of
all \mbox{rhombic-cr} and square sublattices is given by
\begin{equation}
   \Phi_{\lozenge+\square}(s) \, = \, 
   \Phi_{\text{\rm even}}(s)+\Phi_{\text{\rm odd}}(s) \, = \,
   \frac{1-2^{-s}+2^{-2s+1}}{1+2^{-s}}\,
   \zeta(s)^2 \, \Phi^{\mathsf{pr}}_{\square}(s) \ts .
\end{equation}
Via standard arguments involving Moebius inversion (see \cite{BSZ11} 
and references therein),
the number of \emph{primitive} \mbox{rhombic-cr} and square
sublattices together is given by
\begin{equation}
   \Phi_{\lozenge+\square}^{\mathsf{pr}}(s) \, = \, 
   \frac{1}{\zeta(2s)}\, \Phi_{\lozenge+\square}(s) \, = \,
   \frac{1-2^{-s}+2^{-2s+1}}{1+2^{-s}}\,
   \frac{\zeta(s)^2}{\zeta(2s)}\,
   \Phi^{\mathsf{pr}}_{\square}(s).
\end{equation}
Putting all this together, we obtain the generating functions
$\Phi^{\mathsf{pr}}_{\square}$,  $\Phi_{\lozenge}^{\mathsf{pr}}$ and 
$\Phi_{\rect}^{\mathsf{pr}}$ for the number of primitive square, 
\mbox{rhombic-cr} and rectangular sublattices, respectively, as
\begin{align}
  \Phi^{\mathsf{pr}}_{\square}(s)
     &\, = \, (1+2^{-s})\prod_{p\equiv 1(4)}\frac{1+p^{-s}}{1-p^{-s}}
      \, = \, \frac{\zeta(s) \ts L(s,\chi^{}_{-4})}{\zeta(2s)} 
      \ts , \\[1mm]
  \Phi_{\lozenge}^{\mathsf{pr}}(s)
     &\, = \, \left(\frac{1-2^{-s}+2^{-2s+1}}{1+2^{-s}}\,
       \frac{\zeta(s)^2}{\zeta(2s)}-1\right)
       \Phi^{\mathsf{pr}}_{\square}(s) \ts , \\[3mm]
  \Phi_{\rect}^{\mathsf{pr}}(s)
     &\, = \, \left(\frac{\zeta(s)^2}{\zeta(2s)}-1\right)
        \Phi^{\mathsf{pr}}_{\square}(s) \ts ,
\end{align}
with the $L$-series and the character $\chi^{}_{-4}$ from above (see
Appendix~\ref{app-sec:ana} for details and notation).  Note that the
last equation follows from the fact that the generating function for
all rectangular lattices including the square lattices is given by
$\zeta(s)^2 \Phi^{\mathsf{pr}}_{\square}(s)$.

Let us return to the well-rounded sublattices. Since $z_1$ and $z_2$
are shortest (non-zero) vectors, we have $|z_1\pm z_2|^2\geq|z_1|^2
=|z_2|^2$, which is equivalent to $\min(p^2,q^2)\geq
\frac{p^2+q^2}{4}$, which in turn is equivalent to $3p^2\geq q^2\geq
\frac{1}{3}p^2$. Note that this condition is also sufficient.  Hence,
we have to apply this extra condition to our considerations from
above.  We distinguish two cases:
\begin{enumerate}
\item $p$ and $q$ are both even, $\sqrt{3}p\geq q\geq \frac{1}{\sqrt{3}}p$,
  and $z$ may or may not be divisible by $1+\ii$. We write $p=2p', q=2q'$,
  for which we likewise have $\sqrt{3}p'\geq q'\geq \frac{1}{\sqrt{3}}p'$.
  The index is even since it is given by
  $2p'q'|z|^2$.  Here, $p'$ and $q'$ may take any positive integral values,
  even or odd, which satisfy $\sqrt{3}p'\geq q'\geq
  \frac{1}{\sqrt{3}}p'$. This corresponds to
  $\mathcal{E},\mathcal{E}'$ in Eqs.~(29) and (31) of~\cite{Fuksh1}.
\item $p$ and $q$ are both odd, $\sqrt{3}p\geq q\geq \frac{1}{\sqrt{3}}p$,
  and $z$ is divisible by $1+\ii$. We write $z=(1+\ii)w$.  The index is odd
  and given by $pq|w|^2$. This corresponds to
  $\mathcal{O},\mathcal{O}'$ in Eqs.~(30) and (32) of \cite{Fuksh1}.
\end{enumerate}
The set of all possible indices of well-rounded sublattices is thus given by
(we may interchange $p$ and $q$ if necessary)
\begin{equation}\label{eq:sq-indices}
  \bigl\{ 2pq|z|^2 \, \big| \, q\leq p\leq \sqrt{3}q, z\in\ZZ[\ii] \bigr\}
  \, \cup \, \bigl\{ pq|z|^2 \, \big| \,  q\leq p\leq \sqrt{3}q, 
  z\in\ZZ[\ii], \ts 2\nmid pq|z|^2 \bigr\}
\end{equation}
Note that this set is a proper subset of
Fukshansky's~\cite[Thm~1.2, Thm~3.6]{Fuksh1} index set 
\begin{equation}
   \mathcal{D} \, := \, \bigr\{ pq|z|^2 \, \big| \, q\leq p\leq 
             \sqrt{3}q, z\in\ZZ[\ii] \bigl\}
\end{equation}
since $6=2\cdot 3\cdot |1|^2\in \mathcal{D}$, but $6$ is not contained
in the set (\ref{eq:sq-indices}).

The Dirichlet series generating function for the well-rounded
sublattices may now be calculated as above by taking the condition
$\sqrt{3}p\geq q\geq \frac{1}{\sqrt{3}}p$ into account, so that the
generating Dirichlet series for the well-rounded sublattices of even
index is given by
 \begin{equation}\label{eq:sq-even0}
   \frac{1}{2^s}  \sum_{p\in\NN}\,
   \sum_{\frac{1}{\sqrt{3}} p < q < \sqrt{3}p}
   \frac{1}{p^s q^s} \, \Phi^{\mathsf{pr}}_{\square}(s) \ts .
\end{equation}
Clearly, this sum is symmetric in $p$ and $q$, and comprises the
similar sublattices. In fact, if we exclude the square sublattices
(those lattices with $p=q$) from Eq.~\eqref{eq:sq-even0} and note that
$\sum_{p\in\NN}\sum_{\frac{1}{\sqrt{3}} p < q <
  p}=\sum_{q\in\NN}\sum_{q< p <\sqrt{3}q}$, we obtain the generating
function for the rhombic lattices with even index as
\begin{equation}\label{eq:sq-even}
   \Phi_{\mathsf{wr},\text{\rm even}}(s)\, = \, \frac{2}{2^s}
   \sum_{p\in\NN}\, \sum_{p < q < \sqrt{3}p}\frac{1}{p^s q^s}
   \, \Phi^{\mathsf{pr}}_{\square}(s) \ts .
\end{equation}
The case of odd indices is slightly more cumbersome. Here, we have to
replace the factor $(1-2^{-s})^2\zeta(s)^2$ by the corresponding sum
over all odd integers with $p < q < \sqrt{3}p$. Writing
$p=2k+1$ and $q=2\ell+1$, our condition reads $k < \ell <
\sqrt{3}k+\frac{\sqrt{3}-1}{2}$. Since this inequality has no integral
solution for $k=0$, we may start our sum with $k=1$, and finally
arrive at
\begin{equation}\label{eq:sq-odd}
  \Phi_{\mathsf{wr}, \text{\rm odd}}(s) \, = \,
  \frac{2}{1+2^{-s}}\, \Phi^{\mathsf{pr}}_{\square}(s)
  \sum_{k\in\NN}\, \sum_{k < \ell < \sqrt{3}k+\frac{\sqrt{3}-1}{2}}
  \frac{1}{(2k+1)^s (2\ell+1)^s} \ts .
\end{equation}
Now, $\Phi_{\mathsf{wr}, \text{\rm even}}(s)+\Phi_{\mathsf{wr},
  \text{\rm odd}}(s)+\Phi^{}_{\square}(s)$ gives the Dirichlet series
generating function $\Phi^{}_{\square,\mathsf{wr}}(s)$ for the
arithmetic function $a^{}_{\square} (n)$ of well-rounded sublattices
of $\ZZ[\ii]$ of index $n$.  To get a better understanding of it, we
`sandwich' it, on the half-axis $s>1$, between two explicitly known
meromorphic functions.  All these Dirichlet series satisfy the
conditions of Theorem~\ref{thm:meanvalues} (see
Appendix~\ref{app-sec:ana}).  This gives a result on the asymptotic
growth and its error as follows.
\begin{theorem}\label{thm:sq1}
  Let\/ $a^{}_{\square} (n)$ be the number of well-rounded sublattices
  of index\/ $n$ in the square lattice, and\/
  $\Phi^{}_{\square,\mathsf{wr}}(s) =\sum_{n=1}^\infty
  a^{}_{\square}(n)n^{-s}$ the corresponding Dirichlet series
  generating function. The latter is given by
\[
   \Phi^{}_{\square,\mathsf{wr}}(s) \, = \,
   \Phi^{}_{\square} (s) + \Phi^{}_{\mathsf{wr},\text{\rm even}} (s) +
   \Phi^{}_{\mathsf{wr},\text{\rm odd}} (s)
\]
via Eqs.~\eqref{eq:Dedekind-square}, \eqref{eq:sq-even}
and~\eqref{eq:sq-odd}.  The generating function\/
$\Phi^{}_{\square,\mathsf{wr}}$ is meromorphic in the half plane
\mbox{$\{\RE(s) > \frac{1}{2}\}$}, with a pole of order $2$ at $s=1$,
and no other pole in the half plane \mbox{$\{\RE(s) \ge 1\}$}.

If $s>1$, we have the inequality
\[
    D^{}_{\square} (s) - \Phi^{}_{\square} (s) \, < \, 
    \Phi^{}_{\square,\mathsf{wr}}(s) \, < \,
    D^{}_{\square} (s) + \Phi^{}_{\square} (s) \ts ,  
\]
with\/ $\Phi^{}_{\square} (s)$ from Eq.~\eqref{eq:Dedekind-square} and the
function
\[
    D^{}_{\square} (s) \, = \,
    \myfrac{2 + 2^{s}}{1 + 2^{s}} \,
    \frac{1-\sqrt{3}^{1-s}}{s-1} \,
    \frac{L(s,\chi^{}_{-4})}{\zeta (2s)} 
    \, \zeta (s)\ts \zeta (2s-1) \ts . 
\]
As a consequence, the summatory function\/ $A^{}_{\square} (x) =
\sum_{n\leq x} a^{}_{\square} (n)$ possesses the asymptotic growth
behaviour
\[
    A^{}_{\square} (x) \, = \, \myfrac{\log(3) }{2\ts \pi} \,
      x \log(x)  + \oo\bigl(x \log(x)\bigr) ,
   \quad \text{as } x \to\infty\ts .
\]
\end{theorem}

\begin{proof}
  Clearly, $\Phi^{}_{\square,\mathsf{wr}}(s)$ is the sum of
  $\Phi^{}_{\square} (s)$ and the two contributions from
  Eqs.~\eqref{eq:sq-even} and~\eqref{eq:sq-odd}.  For real $s>1$, the
  latter can be both bounded from below and above by an application of
  Lemma~\ref{lem:est-sum} from Appendix~\ref{app-sec:ana} with $\alpha
  = \sqrt{3}$, the former with parameters $\beta = \gamma = 0$ and the
  latter (after pulling out a factor of $2^s$ in the denominator) with
  $\beta = (\sqrt{3} - 1)/2$ and $\gamma = \frac{1}{2}$. A
  straight-forward calculation leads to the explicit expression for
  the function $D^{}_{\square} (s)$, as well as to the inequality
  stated.
  
  It follows from the explicit expression for $D^{}_{\square} (s)$
  that it is a meromorphic function in the whole plane.  Using the
  Euler summation formula, we see that the difference
  $\Phi^{}_{\square,\mathsf{wr}}(s)-D^{}_{\square} (s)$ is an analytic
  function for $\RE(s) > \frac{1}{2}$, guaranteeing that
  $\Phi^{}_{\square,\mathsf{wr}}(s)$ is meromorphic in the half plane
  \mbox{$\{\RE(s) > \frac{1}{2} \}$}.

  The right-most singularity of $\zeta (s) \zeta(2s-1)$ is $s=1$, with
  a pole of the form $\frac{1}{2 (s-1)^2}$, while the entire factor of
  $D^{}_{\square} (s)$ in front of it is analytic near $s=1$ (as well
  as on the line $\{ \RE (s) = 1 \}$).  An application of
  Theorem~\ref{thm:meanvalues} from Appendix A now leads to the
  claimed growth rate.
\end{proof}

The difference of the bounds in Theorem~\ref{thm:sq1} is $2
\Phi^{}_{\square} (s)$, which is a Dirichlet series that itself allows
an application of Theorem~\ref{thm:meanvalues}.  The corresponding
summatory function has an asymptotic growth of the form $c x +
\oo(x)$, which suggests that the error term of $A^{}_{\square} (x)$
might be improved in this direction. However, it seems difficult to
extract good error terms from Delange's theorem; compare the example
in~\cite[Sec 1.8]{Bruedern}.  Since numerical calculations support the
above suggestion, we employed direct methods such as Dirichlet's
hyperbola method; compare~\cite[Sec~3.5]{Apostol}
or~\cite[Sec. I.3]{Tenenbaum}.  A lengthy calculation
(see~\cite{suppl} for the details) finally leads to the following
result.

\begin{theorem}\label{thm:sq2}
  Let\/ $a^{}_{\square} (n)$ be the number of well-rounded sublattices
  of index\/ $n$ in the square lattice.  Then, the summatory function\/
  $A^{}_{\square} (x) \, = \, \sum_{n\leq x} a^{}_{\square} (n)$
  possesses the asymptotic growth behaviour
\begin{align*}
  A^{}_{\square} (x) \, & = \, \frac{\log(3) }{3}
      \frac{L(1,\chi^{}_{-4})}{\zeta(2)} x (\log(x) - 1)
      + c^{}_{\square} x + \OO\bigl(x^{3/4}\log(x)\bigr) \\[1mm]
  & = \, \frac{\log(3) }{2\pi} x \log(x) + \left( c^{}_{\square} -
    \frac{\log(3) }{2\pi} \right) x + \OO\bigl(x^{3/4}\log(x)\bigr)
  \nn
\end{align*}
where, with $\gamma$ denoting the Euler--Mascheroni constant,
\begin{align*}
  c^{}_{\square}& := \frac{L(1,\chi^{}_{-4})}{\zeta(2)} \Biggl(
  \zeta(2) + \frac{\log(3)}{3} \left(
    \frac{L'(1,\chi^{}_{-4})}{L(1,\chi^{}_{-4})} + \gamma -
    2\frac{\zeta'(2)}{\zeta(2)} \right) +
  \frac{\log(3)}{3}\left(2\gamma - \frac{\log(3)}{4} -
    \frac{\log (2)}{6} \right) \\[1mm]
  & \quad - \sum_{p=1}^\infty \frac{1}{p} \biggl(\frac{\log(3)}{2} -
  \sum_{p<q< p\sqrt{3}} \frac{1}{q} \biggr)  
  - \frac{4}{3}\sum_{k=0}^\infty \frac{1}{2k+1}\biggl(
  \frac{1}{4}\log(3) - \sum_{k<\ell< k\sqrt{3}+(\sqrt{3}-1)/2}
  \frac{1}{2\ell+1} \biggr) \Biggr) \\
  & \approx 0.6272237 
\end{align*}
is the coefficient of\/ $(s-1)^{-1}$ in the Laurent series of\/
$\sum_{n \geq 1} a^{}_{\square} (n) n^{-s}$ around\/ $s=1$.
\end{theorem}

Note that $L'(1,\chi^{}_{-4})$ can be computed efficiently via
\begin{equation}
\frac{L'(1,\chi^{}_{-4})}{L(1,\chi^{}_{-4})} 
    \, = \, \log\left(M(1,\sqrt{2})^2 \, \frac{e^\gamma}{2}\right)
    \, = \, \log\biggl(\Gamma\biggl(\frac{3}{4}\biggr)^{\! 4} 
    \, \frac{e^\gamma}{\pi}\biggr)
    \, \approx \, 0.2456096 \ts ,
\end{equation}
where $M(x,y)$ is the arithmetic-geometric mean of $x$ and $y$, and
$\Gamma$ denotes the gamma function; see~\cite{Moree} and references
therein.

\begin{proof}[Sketch of proof]
  $\Phi^{}_{\square,\mathsf{wr}}(s)=\sum_{n=1}^\infty
  a^{}_{\square}(n)n^{-s}$ is a sum of three Dirichlet series, each of
  which is itself a product of several Dirichlet series. Hence, each
  contribution to $a^{}_{\square}(n)$ is a Dirichlet convolution of
  arithmetic functions. The asymptotic behaviour can thus be
  calculated by elementary methods as described in
  \cite[Sec. 3.5]{Apostol}, making use of Euler's summation
  formula~\eqref{eq:euler-sum} wherever appropriate.  To be more
  specific, let
\begin{equation}
  \Phi_{\mathsf{wr},\text{\rm even}}(s)\, =
   \sum_{n\in\NN}\frac{a_{\text{\rm even}}(n)}{n^s} \ts ,
\end{equation}
which is a product of the Dirichlet series
\begin{align}
  \frac{2}{2^s}\frac{1}{\zeta(2s)}&\, =\sum_{n\in\NN}
         \frac{c(n)}{n^s} \ts , \nn\\
  \sum_{p\in\NN}\, \sum_{p < q < \sqrt{3}p}\frac{1}{p^s q^s}
          &\, =\sum_{n\in\NN}\frac{w(n)}{n^s} \ts , \nn\\
  \Phi^{}_{\square}(s)&\, =\sum_{n\in\NN}\frac{b(n)}{n^s} \ts . \nn
\end{align}
Hence $a_{\text{\rm even}}=c*w*b$ is the Dirichlet convolution of
$c,w,b$. The summatory function of a Dirichlet convolution $f*g$ can
now be calculated via the classic formulas (compare~\cite{Apostol}
and~\mbox{\cite[Sec. I.3.2]{Tenenbaum})}
\begin{align}\label{eq:diri-conv1}
  \sum_{n\leq x} \left(f*g\right)(n)&
     \, =\sum_{m\leq x}\;\sum_{d\leq x/m} f(m)g(d)\\
     &\, =\sum_{m\leq \sqrt{x}}\;\sum_{m<d\leq x/m} 
      \bigl(f(m)g(d)+f(d)g(m)\bigr)
      +\sum_{m\leq \sqrt{x}}f(m)g(m) \ts , 
      \label{eq:diri-conv2}
\end{align}
where the latter formula is used for the convolutions $w*b$ and
$b=\chi^{}_{-4}*1$.
\end{proof}

\section{Well-rounded sublattices of $\ZZ[\rho]$}\label{sec:tr}

Next, we consider the hexagonal lattice $\ZZ[\rho]$, with
$\rho=\frac{1+\ii \sqrt{3}}{2}$. As an arithmetic object, it is the
ring of Eisenstein integers, the maximal order of the quadratic field
$\QQ(\ii\sqrt{3}\, )$.  The Dirichlet series generating function for the
number of similar sublattices of $\ZZ[\rho]$ is
\begin{equation}\label{eq:Dedekind-tri}
   \Phi^{}_{\triangle} (s) \, = \, \zeta^{}_{\QQ (\rho)} (s)
   \, = \, L (s,\chi^{}_{-3}) \ts \zeta (s) \ts ,
\end{equation}
with the character
\[
    \chi^{}_{-3} (n) \, = \,
    \begin{cases} 0 , & \text{if $n \equiv 0 \bmod 3$, } \\
    1 , & \text{if $n\equiv 1 \bmod 3$, } \\
    -1 , & \text{if $n\equiv 2 \bmod 3$, }
    \end{cases}
\]
see \cite{BSZ11, Zagier} and Appendix~\ref{app-sec:ana}.

Let $\{ z^{}_{1}, z^{}_{2} \}$ be a reduced basis of a well-rounded
sublattice of $\ZZ[\rho]$.  The orthogonality of $z_1+z_2$ and
$z_1-z_2$ implies that $\frac{z_1+z_2}{z_1-z_2}=\ii\sqrt{3}\ r$ with
$r\in\QQ$.  This shows that square lattices cannot occur here since
this would require $|z_1+z_2|^2=|z_1-z_2|^2$, which is impossible.
Thus, the well-rounded sublattices of $\ZZ[\rho]$ are
\mbox{rhombic-cr} or hexagonal lattices. However, at least one of
$z_1+z_2$ and $z_1-z_2$ is divisible by $\ii\sqrt{3}=\rho-\bar{\rho}$,
and w.l.o.g.\ we may assume that $\ii\sqrt{3}$ divides
$z_1-z_2$. Hence, there exist $p$ and $q\in\ZZ$ together with a
primitive $z\in\ZZ[\rho]$ such that $z_1+z_2=pz$ and
$z_1-z_2=\ii\sqrt{3}qz$. Here, primitive means that $n=1$ is the only
integer $n\in \NN$ that divides $z$. We may again choose $p$ and $q$
positive and
\begin{equation}
  \textstyle\Gam \, = \,\langle z_1,z_2\rangle_\ZZ \, = \, 
  \left\langle \frac{p+\ii\sqrt{3} q}{2}z, 
    \frac{p-\ii\sqrt{3} q}{2}z\right\rangle_\ZZ
    = \, \left\langle (\frac{p-q}{2}+\rho q)z, 
   (\frac{p+q}{2}-\rho q)z\right\rangle_\ZZ
\label{eq:rhombic-crsublattice}
\end{equation} 
is thus a sublattice of index $pq|z|^2$. In particular, $\Gam$ is
a hexagonal lattice if and only if $p=q$ or $p=3q$. Note that
Eq.~(\ref{eq:rhombic-crsublattice}) shows that $p$ and $q$ have the same
parity.

Well-rounded sublattices must satisfy the additional constraints
$|z_1\pm z_2|^2\geq|z_1|^2=|z_2|^2$, which, in this case, are
equivalent to $q\leq p \leq 3q$. The set of possible indices of
well-rounded sublattices is thus given by
\begin{equation}
  \bigl\{ 4pq|z|^2 \, \big| \,q\leq p\leq 3q, z\in\ZZ[\rho] \bigr\} 
  \, \cup \, 
  \bigl\{ pq|z|^2 \, \big| \, q\leq p\leq 3q, z\in\ZZ[\rho], 
  \ts 2\nmid pq \bigr\} \ts .
\end{equation}
An alternative parametrisation of this set can be found
in~\cite[Cor.~4.9]{Fuksh3}. The equivalence of these formulations can
easily be checked by recalling that the (rational) primes represented
by the norm form $m^2-mn+n^2$ of $\ZZ[\rho]$ are precisely $3$ and all
primes $p\equiv 1\pmod 3$.

Counting the number of distinct well-rounded sublattices of a given
index works essentially as in the square lattice case. However, we
have to avoid counting the same lattice twice. Let $z$ be divisible by
$\ii\sqrt{3}$, so that $z=\ii\sqrt{3}w$. Then,
\begin{align}
  z_1&\, = \, \frac{p+\ii\sqrt{3} q}{2}z
        \, = \, -\frac{3q-\ii\sqrt{3} p}{2}w \ts ,\\
  z_2&\, = \, \frac{p-\ii\sqrt{3} q}{2}z
        \, = \, \frac{3q+\ii\sqrt{3} p}{2}w
\end{align}
shows that the tuples $(p,q,z)$ and $(3q,p,w)$ correspond to the same
sublattice. Thus, we only sum over primitive $z$ that are not divisible by
$\ii\sqrt{3}$.

Since we know the generating function~\eqref{eq:Dedekind-tri} for the
similar sublattices already from~\cite{Baake+Grimm}, we concentrate on
the rhombic sublattices here (excluding hexagonal sublattices,
as before).  The
summation over all primitive $z\in\ZZ[\rho]$ not divisible by
$\ii\sqrt{3}$ gives the contribution
$\frac{1}{1+3^{-s}}\Phi^{\mathsf{pr}}_{\triangle}(s)$. The generating
function of all rhombic sublattices of even index then reads
\begin{equation}\label{eq:tri-even}
   \Phi^{}_{\triangle,\mathsf{wr},\text{\rm even}}(s) 
   \, = \, \frac{3}{4^s(1+3^{-s})}
   \sum_{p\in\NN}\ \sum_{p < q < 3 p}\frac{1}{p^s q^s}\,
   \Phi^{\mathsf{pr}}_{\triangle}(s) \ts ,
\end{equation}
where the factor of $3$ reflects that each sublattice occurs in three
different orientations.

In the case of odd indices, we substitute again $p=2k+1$ and
$q=2\ell+1$, wherefore our constraints read $k < \ell < 3k+1$.  This
leads to the following expression for the generating function of all
rhombic sublattices of odd index:
\begin{equation}\label{eq:tri-odd}
  \Phi_{\triangle,\mathsf{wr}, \text{\rm odd}}(s) \, = \, \frac{3}{1+3^{-s}}
  \sum_{k\in\NN}\ \sum_{k < \ell < 3k+1}
  \frac{1}{(2k+1)^s (2\ell+1)^s}\,
  \Phi^{\mathsf{pr}}_{\triangle}(s) \ts .
\end{equation}

Now, we can apply the same strategy as in the square lattice case.

\begin{theorem}\label{thm:tri1}
  Let\/ $a^{}_{\triangle} (n)$ be the number of well-rounded
  sublattices of index\/ $n$ in the hexagonal lattice, and\/
  $\Phi^{}_{\triangle,\mathsf{wr}} (s) = \sum_{n=1}^{\infty}
  a^{}_{\triangle} (n)n^{-s}$ the corresponding Dirichlet series
  generating function. It is given by
\[
  \Phi^{}_{\triangle,\mathsf{wr}} (s) \, = \,
   \Phi^{}_{\triangle} (s) + \Phi^{}_{\triangle,\mathsf{wr}, \text{\rm even}}
    (s) + \Phi^{}_{\triangle,\mathsf{wr},\text{\rm odd}} (s) \ts ,
\]
with the series from Eqs.~\eqref{eq:Dedekind-tri}, \eqref{eq:tri-even}
and~\eqref{eq:tri-odd}.

If\/  $s>1$, we have the inequality
\[
   D^{}_{\triangle} (s) - E^{}_{\triangle} (s) \, < \, 
   \Phi^{}_{\triangle,\mathsf{wr}} (s) \, < \,
   D^{}_{\triangle} (s) \ts,
\]
   with the functions
\[ 
\begin{split}
   D^{}_{\triangle} (s) & \, =\, \myfrac{1}{2}\myfrac{3}{1 + 3^{-s}}\,
   \myfrac{1 - 3_{}^{1-s}}{s-1}\, \frac{L(s,\chi^{}_{-3})}{\zeta (2s)}
   \, \zeta (s)\ts \zeta (2s-1) \, , \\
   E^{}_{\triangle} (s) & \, = \, 
   \myfrac{3}{1+3^{-s}} \, L(s,\chi^{}_{-3}) \ts \zeta(s)\ts . 
\end{split}
\]
The function\/ $\Phi^{}_{\triangle,\mathsf{wr}}(s)$ is meromorphic in
the half plane \mbox{$\{\RE(s) > \frac{1}{2}\}$}, with a pole of
order\/ $2$ at\/ $s=1$, and no other pole in the half plane
\mbox{$\{\RE(s) \ge 1\}$}.  As a consequence, the summatory function\/
$A^{}_{\triangle} (x) \, = \sum_{n\leq x} a^{}_{\triangle} (n)$, as $x\to\infty$,
possesses the asymptotic growth behaviour
\[
         A^{}_{\triangle} (x) \, = \,
         \myfrac{3 \sqrt{3} \ts \log(3)}{8 \ts \pi}\,
          x\log(x)  + \oo\bigl(x\log(x)\bigr).
\]
\end{theorem}
\begin{proof}[Sketch of proof]
  In analogy to before, $\Phi^{}_{\triangle,\mathsf{wr}} (s)$ is the
  sum of the contributions from Eqs.~\eqref{eq:tri-even}
  and~\eqref{eq:tri-odd}. The calculation of the upper and lower
  bounds can be done as in Theorem~\ref{thm:sq1} via
  Lemma~\ref{lem:est-sum}, this time with $\alpha = 3$ and appropriate
  choices for $\beta$ and $\gamma$. The conclusion on the growth rate
  of $A^{}_{\triangle}(x)$ follows as before from
  Theorem~\ref{thm:meanvalues}.
\end{proof}

As for the square lattice, we can improve the error term considerably
by lengthy but elementary calculations (see~\cite{suppl} for the
details).  Eventually, we obtain the following result.

\begin{theorem}\label{thm:tri2}
  Let\/ $a^{}_{\triangle} (n)$ be the number of well-rounded
  sublattices of index\/ $n$ in the hexagonal lattice.  Then, the
  summatory function\/ $A^{}_{\triangle} (x) \, =  \sum_{n\leq x}
  a^{}_{\triangle} (n)$ possesses the asymptotic growth behaviour
\begin{align*}
  A^{}_{\triangle} (x) 
    & \, = \,  \frac{9\log(3)}{16}\,
      \frac{L(1,\chi^{}_{-3})}{\zeta(2)} \,
      x (\log(x) - 1)
              + c^{}_\triangle x + \OO\bigl(x^{3/4}\log(x)\bigr) \\[1mm]
    & \, = \, \frac{3\sqrt{3}\, \log(3)}{8\pi} x (\log(x) - 1)
              + c^{}_\triangle x + \OO\bigl(x^{3/4}\log(x)\bigr) , 
\end{align*}
where
\begin{align*}
  c^{}_\triangle& \, = \, L(1,\chi^{}_{-3}) + \frac{9 \log(3)
    L(1,\chi^{}_{-3})}{16 \zeta(2)}\Biggl( \biggl( \gamma +
  \frac{L'(1,\chi^{}_{-3})}{L(1,\chi^{}_{-3})} - 2
  \frac{\zeta'(2)}{\zeta(2)}\biggr)
  + 2\gamma - \frac{\log(3)}{4}  \\[1mm]
  & \qquad - \sum_{p=1}^\infty \frac{1}{p} \biggl(\log(3) -
  \sum_{p<q\leq 3p-1} \frac{1}{q} \biggr) - \sum_{k=0}^\infty
  \frac{4}{2k+1}\biggl( \frac{1}{2}\log(3) - \sum_{k<\ell\leq 3k}
  \frac{1}{2\ell+1}\biggr)
  \Biggr)   \\
  & \, \approx \, 0.4915036 
\end{align*}
is the coefficient of\/ $(s-1)^{-1}$ in the Laurent series of\/
$\sum_n\frac{a^{}_{\triangle} (n)}{n^s}$ around\/ $s=1$.   \qed
\end{theorem}

The number $L'(1,\chi^{}_{-3})$ can be computed efficiently as well,
via a formula involving the arithmetic-geometric mean (see~\cite{Moree}), 
and reads
\begin{equation}
\frac{L'(1,\chi^{}_{-3})}{L(1,\chi^{}_{-3})} 
    \, = \, \log\left(\frac{2_{}^{\frac{3}{4}}\ts M \! \left(1,
          \cos (\frac{\pi}{12})
          \right)^2 \,   e^\gamma}{3}\right)
    \, = \, \log\left(\frac{2^4 \pi^4\ts e^\gamma}{
           3^{\frac{3}{2}}\, \Gamma\!\left(\frac{1}{3}\right)^6}\right)
    \, \approx \, 0.3682816 \ts .
\end{equation}

Above and in the previous section, we have seen that the asymptotic
growth rate for the hexagonal and square lattice is of the form $c_1
x\log(x) + c_2 x + \OO\bigl(x^{3/4}\log(x)\bigr)$.  Actually,
numerical calculations suggest that the error term is $\OO(x^{1/2})$
or maybe even slightly better.

Let us now see what we can say about the other planar lattices.

\section{The general case}\label{sec:gen}

\subsection{Existence of well-rounded sublattices}\label{sub:existence}

Recall from Section~\ref{sec:tools} that a lattice allows a
well-rounded sublattice if and only if it contains a rectangular or
square sublattice. The following lemma contains several reformulations
of this property.
\begin{lemma}\label{lem:BRS}
  Let\/ $\Gam$ be any planar lattice. There are natural bijections
  between the following objects:
\begin{enumerate}
\item \emph{Rational orthogonal frames} for $\Gam$, that is, unordered
  pairs $\QQ w, \QQ z$ of perpendicular ($w \bot z$), one-dimensional
  subspaces of the rational space $\QQ \Gam$ generated by $\Gam$ (so
  we may assume $w, z \in \Gam$).
\item Unordered pairs $\{\pm R\}$ of coincidence reflections of
  $\Gam$; from now on, we shall simply write $\pm R$ for such a pair.
\item \emph{Basic} rectangular or square sublattices $\Lam \sbe \Gam$,
  where `basic' means that $\Lam = \langle w, z \rangle _\ZZ$ with $w,
  z$ primitive in $\Gam$ (so $\QQ w \cap \Gam = \ZZ w$ and $\QQ z \cap
  \Gam = \ZZ z$). We shall call them \emph{BRS sublattices} for short.
\item Four-element subsets $\{\pm w, \pm z\} \subset \Gam$ 
  of non-zero primitive
  lattice vectors with $w \perp z$.
\end{enumerate}
\end{lemma}

Given $\Gam$, we use the notation $\CR = \CR_\Gam$ for the set of all
pairs $\pm R$ of coincidence reflections of $\Gam$. So $\CR_\Gam$ is
in natural bijection with any of the four sets described in
Lemma~\ref{lem:BRS}.  For the rest of the paper, we introduce the
following notation, based on Lemma~\ref{lem:BRS}.  For $\pm R \in
\CR_\Gam$, we denote by $\Gam_R$ (rather than $\Gam_{\pm R}$) the
corresponding BRS sublattice. Explicitly, this is
\[
\begin{split}
  \Gam_R & \, = \, \Gam \cap \Fix (R) \oplus \Gam \cap \Fix (-R) \\
         & \, = \, \ZZ w \oplus \ZZ z, \; 
           \text{ where } Rw = w,\, Rz = -z
\end{split}
\]
(thus $w, z$ are primitive in $\Gam$). In accordance with part (2) of
Lemma~\ref{lem:BRS}, we have $\Gam_{R}=\Gam_{-R}$, with the roles of
$w$ and $z$ interchanged.  If we start with an arbitrary primitive
vector $w \in \Gam$, we similarly write
\[ 
   \Gam_w \, := \, \ZZ w \oplus \ZZ z, \; \text{ where } z 
  \perp w,\, z  \text{ primitive in } \Gam.
\]
The four element set $\{\pm w, \pm z\}$ is uniquely determined by any
of its members, and $\Gam_w$ is the unique BRS-sublattice belonging to
this set, according to part (4) of the remark.

In addition to $\Gam_R$, there is a second sublattice of $\Gam$ which
is invariant under $R$ and contains $w, z$ as primitive vectors. This
is
\begin{equation}\label{eq:gamma-tilde}
  \widetilde \Gam _R \, := \, \Bigl\langle \frac{w+z}{2}, 
   \frac{w-z}{2}\Bigr\rangle_{\ZZ} \, ,
\end{equation}
the unique superlattice of $\Gam_R$ containing $\Gam_R$ with index $2$
in such a way that $w,z$ are still primitive in $\widetilde \Gam _R$.
By the way, it is a purely algebraic fact that, if $R$ is a non-trivial
automorphism of order $2$ of an abstract lattice $\Lam$ (free
$\ZZ$-module) of rank $2$, i.e. $R^2 = \id \ne \pm R$, then either
$\Lam$ has a $\ZZ$-basis $w,z$ of eigenvectors of $R$ (so $Rz=z,\,
Rw=-w$), or $\Lam$ possesses a $\ZZ$-basis $u,v$ with $Ru=v$. Thus,
already on the level of abstract reflections, one can distinguish
between `rectangular type' and `rhombic type' of a reflection acting
on a lattice. In the situation considered above, the reflection $R$ on
$\Gam_R$ is of rectangular type, and the lattice $\Gam_R$ itself thus
of rectangular or square Bravais type, whereas the reflection $R$ on
$\widetilde \Gam_R$ is of rhombic type, which implies that $\widetilde
\Gam_R$ is of rhombic-cr, square or hexagonal Bravais type. The
significance of $\widetilde \Gam_R$ is explained by the following
lemma.
\begin{lemma}\label{lem:R-over-lattices}
  Given $\Gam$ and $\pm R \in \CR_\Gam$ as above, let $\Lam \supseteq
  \Gam_R = \langle w,z\rangle$ be an $R$-invariant superlattice
  containing\/ $w,z$ as primitive vectors. Then, either $\Lam = \Gam_R$
  or $\Lam = \widetilde \Gam_R$.
\end{lemma}
\begin{proof}
  Since $z$ is primitive, $\Lam$ has a $\ZZ$-basis $u,z$, where $u$ is
  of the form $u = \frac{1}{m}w + \frac{k}{m}z$ with $m = [\Lam :
  \Gam_R]$ and $ 0 \le k <m$. The condition $Ru \in \Lam$ immediately
  leads to $m \in \{1,2\}$ and $k \in\{0, 1\}$, respectively.
\end{proof}
\begin{lemma}\label{lem:even-index}
  Given $\Gam$ and $\pm R \in \CR_\Gam$ as above, $\widetilde \Gam_R$
  is contained in $\Gam$ if and only if the index $[\Gam : \Gam_R]$ is
  even.
\end{lemma}
\begin{proof}
  If $[\Gam : \Gam_R] = [\Gam : \langle w,z \rangle] $ is even and
  $\frac{1}{2} (aw+bz)$ with $ a,b \in \{0,1\}$ represents an element
  of order $2$ in the factor group $\Gam / \Gam_R$, then, since $w/2,
  z/2 \notin \Gam$, we must have $a=b=1$, leading to the sublattice
  $\widetilde \Gam_R$. The converse is clear.
\end{proof}
\begin{coro}\label{cor:even-index}
  For any pair of coincidence reflections $\pm R \in \CR_\Gam$, the
  coincidence site lattice $\Gam (R) = \Gam \cap R\Gam$ is equal to
  $\Gam_R$ or to $\widetilde \Gam_R$. The latter occurs if and only if
  the index\/ $[\Gam : \Gam_R]$ is even. \qed
\end{coro}

The following basic result partitions the set of all planar lattices
admitting a well-rounded (or rectangular) sublattice into two disjoint
classes, as announced at the end of the introduction. Clearly, a
rational lattice possesses infinitely many BRS sublattices, since for any
non-zero lattice vector $v$, the orthogonal subspace of $v$ also
contains a non-zero lattice vector (simply by solving a linear
equation with rational coefficients). 
In contrast, the non-rational case can be analysed as follows.
\begin{prop}\label{prop:wr-exist}
  Let\/ $\Gam$ be non-rational planar lattice which possesses a
  rectangular sublattice, so that\/ $\CR_\Gam \ne \varnothing$ by
  Lemma~$\ref{lem:BRS}$. Then, $|\CR_\Gam | =1$, whence $\Gam$
  possesses exactly one BRS sublattice, and one pair of coincidence
  reflections.
\end{prop}
\begin{proof}
  $\Gam$ has a sublattice $\Lam$ with an orthogonal basis $v,w$, where
  we may assume $|v|=1$ and $|w|^2 = c > 0$. Now assume that there is
  a further vector $u=rv+sw$ with $rs \ne 0$ admitting an orthogonal,
  non-zero vector $u' = r'v + s'w$. Then, $rr' + c ss' = 0$ and
  necessarily $s' \ne 0$, thus $c = -rr'/ss' \in \QQ$. Therefore
  $\Lam$, and thus also $\Gam$, is rational.
\end{proof}
The previous result (with a slightly more complicated proof) is also
found in \cite{Kuehn}, Lemma 2.5 and Remark 2.6. Our approach suggests
the following distinction of cases.

\begin{prop}\label{prop:wr-exist-expl}
  Let\/ $\Gam=\langle 1,\tau\rangle_\ZZ$ be a lattice in\/
  $\RR^2\simeq\CC$, and write\/ $n=|\tau|^2$ and\/
  $t=\tau+\bar{\tau}$. Then, $\Gam$ has a well-rounded sublattice if
  and only if one of the following conditions is satisfied:
\begin{enumerate}
  \item \label{enum:all-rat} 
    $\Gam$ is rational, i.e.\ both $t$ and $n$ are rational;
  \item \label{enum:t-rat} 
    $t$ is rational, but $n$ is not;
  \item \label{enum:t-irr} 
    $t$ is irrational, and there exist $q,r\in\QQ$ with 
    $\sqrt{q+r^2}\in\QQ$ and $n=q+rt$.
\end{enumerate}
\end{prop}
Note that case (\ref{enum:t-irr}) includes both rational and
irrational $n$.  In the case that $n$ is rational, this means that $n$
has to be a rational square.

\begin{proof}
  Recall that $\Gam$ has a well-rounded sublattice if and only if it
  has a rectangular or a square sublattice.  This happens if and only
  if there exist integers $a,b,c,d$ such that the non-zero vectors
  $a+b\tau$ and $c+d\tau$ are orthogonal. The latter condition holds
  if and only if
\begin{equation}\label{eq:orth}
   ac+bd n+(ad+bc)\frac{t}{2} \, = \, 0
\end{equation}
has a non-trivial integral solution, where $n=|\tau|^2$ and
$t=\tau+\bar{\tau}$ are the norm and the trace of $\tau$,
respectively. In fact, there exists an integral solution if and only
if there exists a rational one. This leads to the following three
cases:
\begin{enumerate}
  \item Clearly, Eq.~(\ref{eq:orth}) has a solution if both $t$ and $n$
    are rational.
  \item Let $t\in\QQ, n\not\in\QQ$: 
    Condition (\ref{eq:orth}) is equivalent to $bd=0=ac+(ad+bc)\frac{t}{2}$.
    With $\frac{t}{2}=\frac{p}{q}$, $p,q\in\ZZ$, an integer solution is given
    by $a=1, b=0, c=p, d=-q$.
  \item Let $t\not\in\QQ$, with $n=q+rt$. 
    As $n>0$, at least one of $q$ and $r$
    is non-zero. Here, condition (\ref{eq:orth}) is equivalent to 
    $ac+bdq=0$ and $2bdr+(ad+bc)=0$. As $a=c=0$ would imply
    $a+b\tau=0$ or $c+d\tau=0$, we may assume w.l.o.g.\ that $a\neq 0$. 
    This gives $c=-\frac{bdq}{a}$ and
    $1+2\frac{b}{a}r-\left(\frac{b}{a}\right)^2 q=0$, where we have assumed
    $d\ne 0$ in the latter equation, since otherwise $c+d\tau=0$.
    The latter has a rational solution if and only if
    $r^2+q$ is a square. 
\end{enumerate}
Finally, we have to check that the remaining case does not allow for integral
solutions. Let $t$ and $n$ be irrational and assume that they are
independent over $\QQ$. This clearly requires
$ac=bd=ad+bc=0$, which implies $a+b\tau=0$ or $c+d\tau=0$.
\end{proof}

\begin{remark}\label{rem:kuehn}
  After we had arrived at Proposition~\ref{prop:wr-exist-expl}, we
  became aware of an essentially equivalent result by K\"uhnlein
  \cite[Lemma 2.5]{Kuehn}, where the invariant $\delta(\Gam) = \dim
  \langle 1, t, n \rangle_\QQ$ is introduced.  Clearly, condition (1)
  of Proposition~\ref{prop:wr-exist-expl} is equivalent with
  $\delta(\Gam) = 1$, and our conditions (2) and (3) are equivalent
  with $\delta(\Gam) = 2$ together with the condition that
  K\"uhnlein's `strange invariant' $\sigma(\Gam)$ is the class of all
  squares in $\QQ^\times$.  Here, $\sigma(\Gam)$ is the square class
  of $-\det (X)$, where $X=\left(\begin{smallmatrix} x & y \\ y &
      z \end{smallmatrix}\right)$ is a non-trivial integral matrix
  satisfying $\tr (X G) =0$, with $G=\bigl(\begin{smallmatrix}1 &
    t/2 \\ t/2 & n \end{smallmatrix}\bigr)$ being the Gram matrix of
  $\Gam$.  Altogether, this shows that our criterion is equivalent to
  K\"uhnlein's.
\end{remark}

In the situation of Proposition~\ref{prop:wr-exist}, let $R$ be the
unique (up to a sign) coincidence reflection and $\Gam_R = \langle w,
z \rangle$ the unique BRS sublattice.  We get all well-rounded
sublattices by considering the rectangular sublattices generated by
$kw$, $\ell z$ with the constraint
\begin{equation}
   k\frac{1}{\sqrt{3}}\frac{|w|}{|z|} \, \leq \, \ell 
   \, \leq \, k\sqrt{3}\ \frac{|w|}{|z|} \ts ,
\end{equation}
whose superlattice $\left\langle \frac{1}{2}kw\pm\frac{1}{2}\ell
  z\right\rangle_\ZZ$ is a sublattice of $\Gam$. The latter requires
that $k$ and $\ell$ have the same parity. By
Lemma~\ref{lem:even-index}, odd values $k,\ell$ occur if and only if
the index $\sigma = \sigma_\Gam := [\Gam : \Gam_R]$ is even.  This
gives the following result.

\begin{prop}\label{prop:wr-non-rational}
  Let\/ $\Gam$ be a lattice that has a well-rounded sublattice and
  assume that $\Gam$ is not rational (cf.\
  Proposition~$\ref{prop:wr-exist}$).  Let $\sigma$ be the index of
  its unique BRS sublattice $\Gam_R$ and $\kappa$ be the ratio of the
  lengths of its orthogonal basis vectors. The generating function for
  the number of well-rounded sublattices then reads as follows.
\begin{enumerate}
  \item If\/ $\sigma$ is odd, one has
    \[
       \Phi^{}_{\Gam, \mathsf{wr}}(s) \, = \, \frac{1}{\sigma^s}
        \, \phi^{}_{\mathsf{wr},\text{\rm even}}(\kappa;s),
    \]
    with
    \[
    \phi^{}_{\mathsf{wr},\text{\rm even}}(\kappa;s) \, = \,
    \frac{1}{2^s} \sum_{k\in\NN}\ \sum_{\frac{\kappa}{\sqrt{3}} k \leq
      \ell \leq \sqrt{3} \, \kappa \, k} \frac{1}{k^s \ell^s}\, .
          \]
  \item If\/ $\sigma$ is even, one has
    \[
      \Phi^{}_{\Gam, \mathsf{wr}}(s) \, = \, 
           \frac{1}{\sigma^s}\,
           \phi^{}_{\mathsf{wr},\text{\rm even}}(\kappa;s)
            + \frac{2^s}{\sigma^s}\,
           \phi^{}_{\mathsf{wr},\text{\rm odd}}(\kappa;s),
    \]
    with\/ $\phi^{}_{\mathsf{wr},\text{\rm even}}(\kappa;s)$ as above and
    \[
    \phi^{}_{\mathsf{wr},\text{\rm odd}}(\kappa;s) \, = 
    \sum_{k\in\NN}\ \sum_{\frac{\kappa}{\sqrt{3}} (k+\frac{1}{2})-
      \frac{1}{2} \leq \ell \leq \sqrt{3} \, \kappa \,
      (k+\frac{1}{2})-\frac{1}{2}} \frac{1}{(2k+1)^s (2\ell+1)^s}\, .
     \]
\end{enumerate}
\end{prop}

\begin{remark}\label{rem:kappa}
  The quantity $\kappa = |w|/|z|$ is unique up to taking its
  inverse. Note that $\phi^{}_{\mathsf{wr},\text{\rm even}}(\kappa;s)=
  \phi^{}_{\mathsf{wr},\text{\rm even}}(\tfrac{1}{\kappa};s)$ and
  $\phi^{}_{\mathsf{wr},\text{\rm odd}}(\kappa;s)=
  \phi^{}_{\mathsf{wr},\text{\rm odd}}(\tfrac{1}{\kappa};s)$. Hence,
  there is no ambiguity in the definition of the generating functions.
\end{remark}

In the cases of the square and hexagonal lattices we have been able to
give lower and upper bounds for the generating functions
$\Phi_{\mathsf{wr}}$.  In a similar way we obtain the following
result.
\begin{remark}\label{rem:phi-irr}
  We have the following inequalities for real $s>1$:
\begin{align}
   D^{}_{\text{\rm even}}(\kappa;s)-E^{}_{\text{\rm even}}(\kappa;s) 
  & \, < \, \phi^{}_{\mathsf{wr},\text{\rm even}}(\kappa;s) 
  \, < \, D^{}_{\text{\rm even}}
   (\kappa;s)+E^{}_{\text{\rm even}}(\kappa;s) \ts , \nonumber \\
  D^{}_{\text{\rm odd}}(\kappa;s)-E^{}_{\text{\rm odd}}(\kappa;s) 
  &\, < \;\ts \phi^{}_{\mathsf{wr},\text{\rm odd}}(\kappa;s)
  \; < \, D^{}_{\text{\rm odd}}(\kappa;s)+
   E^{}_{\text{\rm odd}}(\kappa;s) \ts , \nonumber
\end{align}
with the generating functions
\begin{align}
D^{}_{\text{\rm even}}(\kappa;s) & \, = \,
\frac{1}{2^s}\left(\frac{\sqrt{3}}{\kappa}\right)^{\! s-1} 
\frac{1-3^{1-s}}{s-1}\zeta(2s-1)\ts ,
\nonumber \\
E^{}_{\text{\rm even}}(\kappa;s) & \, = \,
\frac{1}{2^s}\left(\frac{\sqrt{3}}{\kappa}\right)^{\! s} \zeta(2s)\ts ,
\nonumber \\
D^{}_{\text{\rm odd}}(\kappa;s) & \, = \,
\frac{1}{2}\left(\frac{\sqrt{3}}{\kappa}\right)^{\! s-1} 
\frac{1-3^{1-s}}{s-1}\left(1-\frac{1}{2^{2s-1}}\right)\zeta(2s-1)\ts ,
\nonumber \\
E^{}_{\text{\rm odd}}(\kappa;s) & \, = \,
\left(\frac{\sqrt{3}}{\kappa}\right)^{\! s}
\left(1-\frac{1}{2^{2s}}\right) \zeta(2s) \ts .
\nonumber
\end{align}
\end{remark}

Let us now have a closer look at the analytic properties of
$\Phi^{}_{\Gam, \mathsf{wr}}$.  Before formulating the theorem, we
observe that the two cases of Proposition~\ref{prop:wr-non-rational}
can be unified by considering the index $\varSigma := [\Gam :
\Gam(R)]$ of the unique non-trivial CSL in $\Gam$. By
Corollary~\ref{cor:even-index}, $\sigma=\varSigma$ if $\sigma$ is odd
and $\sigma=2\varSigma$ if $\sigma$ is even.  We can now formulate a
refinement of Lemma~3.3 and Corollary~3.4 in \cite{Kuehn} as follows.

\begin{prop}\label{prop:non-rat-cont}
  Let $\Gam$ be a lattice with a well-rounded sublattice and
  assume that $\Gam$ is not rational, so that $\Gam$ has exactly one
  non-trivial CSL. Let $\varSigma$ be its index in $\Gam$.  Then, the
  generating function\/ $\Phi^{}_{\Gam, \mathsf{wr}}$ for the number of
  well-rounded sublattices has an analytic continuation to the open
  half plane $\{\RE(s) > \frac{1}{2} \}$ except for a simple pole
  at $s=1$, with residue $\frac{\log(3)}{4\varSigma}$.
\end{prop}

\begin{proof}
  We proceed in a similar way as in the proof of Theorem~\ref{thm:sq1}
  by applying Euler's summation formula to the inner sum. This shows
  that both $\phi^{}_{\mathsf{wr},\text{\rm
      even}}(\kappa;s)-D^{}_{\text{\rm even}}(\kappa;s)$ and
  $\phi^{}_{\mathsf{wr},\text{\rm odd}}(\kappa;s)-D^{}_{\text{\rm
      odd}}(\kappa;s)$ are analytic in the open half plane $\{\RE(s) >
  \frac{1}{2} \}$. Moreover, the explicit formulas from above show
  that both $D^{}_{\text{\rm even}}(\kappa;s)$ and $D^{}_{\text{\rm
      odd}}(\kappa;s)$ are analytic in the whole complex plane except
  at $s=1$, where they have a simple pole with residue
  $\frac{\log(3)}{4}$ and $\frac{\log(3)}{8}$, respectively.
  Inserting this result into the expressions for $\Phi^{}_{\Gam,
    \mathsf{wr}}(s)$, we compute the residue at $s=1$ to
  $\frac{\log(3)}{4\varSigma}$, where we have used that
  $\sigma=\varSigma$ if $\sigma$ is odd and $\sigma=2\varSigma$ if
  $\sigma$ is even.
\end{proof}

Using similar arguments as in the proofs of Theorems~\ref{thm:sq1}
and~\ref{thm:sq2}, one can derive from
Proposition~\ref{prop:non-rat-cont} the asymptotic behaviour of the
number of well-rounded sublattices as follows.
\begin{theorem}\label{thm:non-rat-growth}
  Under the assumptions of Proposition $\ref{prop:non-rat-cont}$, the
  summatory function $A^{}_{\Gam} (x) \, = \, \sum_{n\leq x}
  a^{}_{\Gam} (n)$ possesses the asymptotic growth behaviour
\[
  A^{}_{\Gam} (x) \, = \,
  \frac{\log(3)}{4\varSigma}x + \OO\bigl(\sqrt{x}\ts \bigr)
\]
  as $x\to\infty$.    \qed
\end{theorem}

\subsection{The rational case}\label{sub:rational}

A rational lattice $\Gam$ contains infinitely many BRS sublattices
$\Gam_R$.  Using the same considerations as in the previous
subsection, for any given pair $\pm R$ we can count the number of
well-rounded sublattices invariant under $\pm R$ (that is, contained
in $\widetilde \Gam_R$). Counting all possible well-rounded
sublattices then amounts to sum over all possible pairs $\pm
R$. However, some care is needed in case of square and hexagonal
lattices.

For convenience, we use the notation $\CR_1 := \{\pm R \mid \widetilde
\Gam_R \not \subseteq \Gam \}$ and $\CR_2 := \{\pm R \mid \widetilde
\Gam_R \subseteq \Gam \}$, which, by Lemma~\ref{lem:even-index}, is a
partition of $\CR$ into sets of odd and even index of $\Gam_R$, which
is reflected by the indices $1$ and $2$.

\begin{prop}\label{prop:rational-gen}
  Let\/ $\Gam$ be a rational lattice and let\/
  $\Phi^{\triangle}_{\Gam}(s)$ be the generating function of all
  hexagonal sublattices of\/ $\Gam$. Now, for any pair of coincidence
  reflections\/ $\pm R \in \CR_\Gam$, let\/ $\sigma(R) = \mbox{$[\Gam :
    \Gam_R]$}$ and let\/ $\kappa(R)$ be the length ratio of orthogonal
  basis vectors of\/ $\Gam_R$.  Then, the generating function for the
  number of well-rounded sublattices reads 
\begin{align}\label{eq:rational-gen}
  \Phi^{}_{\Gam, \mathsf{wr}}(s) &\, = \sum_{\pm R \in \CR_1}
  \frac{1}{\sigma(R)^s}\,
  \phi^{}_{\mathsf{wr},\text{\rm even}}(\kappa(R);s)\\
  & \quad + \sum_{\pm R \in \CR_2}
  \frac{1}{\sigma(R)^s}\left(\phi^{}_{\mathsf{wr},\text{\rm
        even}}(\kappa(R);s)
    + 2^s \ts \phi^{}_{\mathsf{wr},\text{\rm odd}}(\kappa(R);s)\right)
   \nonumber\\[1mm]
  & \quad -2\ts \Phi^{\triangle}_{\Gam}(s),\nonumber
\end{align}
where\/ $\phi^{}_{\mathsf{wr},\text{\rm even}}(\kappa;s)$ and\/
$\phi^{}_{\mathsf{wr},\text{\rm odd}}(\kappa;s)$ are as in
Proposition~$\ref{prop:wr-non-rational}$.
\end{prop}

Keep in mind that we sum over pairs of coincidence reflections $\pm R$
here.  According to Lemma~\ref{lem:BRS}, we could alternatively sum
over BRS sublattices or rational orthogonal frames. Furthermore, note
that $\Phi^{\triangle}_{\Gam}(s) = 0$ unless $\Gam$ is commensurate to
a hexagonal lattice.

Before proving Proposition~\ref{prop:rational-gen}, let us have a
closer look at some special cases.
\begin{remark}\label{rem:rational-special}
  If $\Gam$ is not commensurate to a square or a hexagonal lattice,
  all well-rounded sublattices are rhombic. Likewise, all CSLs
  $\Gam(R)$ generated by a reflection are either rectangular or
  rhombic-cr. In fact, there exists a bijection between BRS
  sublattices $\Gam_R$ and the corresponding CSLs $\Gam(R)$, which
  implies that the summation in Eq.~\eqref{eq:rational-gen} could be
  carried out over CSLs as well. In particular, $\CR_1 =
  \CR_{\text{rec}} := \{\pm R \mid \Gam(R) \text{ rectangular}\}$ and
  $\CR_2 = \CR_{\text{rh-cr}} := \{\pm R \mid \Gam(R) \text{
    rhombic-cr}\}$ by Lemma~\ref{lem:even-index}.

  The case that $\Gam$ is commensurate to a hexagonal lattice is the
  only one where the additional term $-2\ts \Phi^{\triangle}_{\Gam}(s)$ is
  non-trivial, which compensates for the fact that the sum over $\pm
  R\in \CR_2$ counts every hexagonal sublattice thrice. Here, we do
  not have the bijection between the BRS sublattices $\Gam_R$ and CSLs
  $\Gam(R)$ any more, and the sums cannot be replaced by sums over
  CSLs. Still, we have a characterisation of the sets $\CR_1$ and
  $\CR_2$ via CSLs, namely $\CR_1 = \CR_{\text{rec}} := \{\pm R \mid
  \Gam(R) \text{ rectangular}\}$ and $\CR_2 = \CR_{\text{rh-cr-hex}}
  := \{\pm R \mid \Gam(R) \text{ rhombic-cr or hexagonal}\}$.

  If $\Gam$ is commensurate to a square lattice, no simple
  characterisation of $\CR_1$ and $\CR_2$ via CSLs is possible. This
  is due to the fact that square CSLs may appear both in $\CR_1$ and
  in $\CR_2$.
\end{remark}

\begin{proof}[Proof of Proposition~$\ref{prop:rational-gen}$]
  As indicated above, counting all well-rounded sublattices that are
  invariant under a given pair $\pm R$ (that is, contained in
  $\widetilde \Gam_R$) gives a contribution
\[
   \frac{1}{\sigma(R)^s}\,
   \phi^{}_{\mathsf{wr},\text{\rm even}}\bigl(\kappa(R);s \bigr)
\]
if $\widetilde \Gam_R \not\in \Gam$, and
\[
   \frac{1}{\sigma(R)^s}\ts
    \Bigl(\phi^{}_{\mathsf{wr},\text{\rm even}}\bigl(\kappa(R);s \bigr)
    + 2^s\ts \phi^{}_{\mathsf{wr},\text{\rm odd}}\bigl(\kappa(R);s \bigr)\Bigr)
\]
if $\widetilde \Gam_R \in \Gam$. If $\Gam$ is not commensurate to a
hexagonal or a square lattice, every well-rounded sublattice is of
rhombic type and belongs to a unique pair $\pm R$ of coincidence
reflections. Thus, summing over all pairs $\pm R$ immediately gives
the result in this case.

The situation is more complex for lattices that are commensurate to a
hexagonal or a square lattice, since some well-rounded sublattices may
be of hexagonal or square type, respectively, and hence there may be
more than one pair $\pm R$ of coincidence reflections associated with
it.  The rhombic well-rounded sublattices may still be treated in the
same way as above, but the hexagonal and square sublattices need extra
care.

A hexagonal sublattice corresponds to exactly three pairs of
coincidence reflections. Thus we count the hexagonal lattices thrice
if we sum over all pairs of coincidence reflections, which we
compensate by subtracting the term $2\ts \Phi^{\triangle}_{\Gam}(s)$.

Similarly, a square sublattice $\Lam$ is invariant under two pairs
$\pm R, \pm S$ of coincidence reflections. However, these two pairs
play different roles, as exactly one of these pairs, say $\pm S$, has
eigenvectors which form a reduced basis of $\Lam$.  This implies that
$\Lam$ is only counted in the set of rhombic and square lattices which
emerge from $\Gam_R$.  Hence, we have a unique pair $\pm R$ in this
case as well, and no correction term is needed here.
\end{proof}

\begin{theorem}
  For any rational lattice $\Gam$, the generating function\/ $
  \Phi^{}_{\Gam, \mathsf{wr}}(s)$ has an analytic continuation to the
  half plane $\{ \RE(s) >\frac{1}{2} \}$ except for a pole of order\/
  $2$ at\/ $s=1$.  Hence there exists a constant $c>0$ such that the
  asymptotic growth rate, as $x\to\infty$, is
\[
     A^{}_{\Gam} (x) \, = \sum_{n\leq x} a^{}_{\Gam}(n) 
     \, \sim \, c \ts x \log(x) \ts .
\]
\end{theorem}

\begin{proof}
  We have already shown that $\phi^{}_{\mathsf{wr},\text{\rm
      even}}(\kappa;s)$ and $\phi^{}_{\mathsf{wr},\text{\rm
      odd}}(\kappa;s)$ are analytic in the half plane $\{ \RE(s) >
  \frac{1}{2} \}$ except for $s=1$, where both functions have a simple
  pole. The same holds true for $\Phi^{\triangle}_{\Gam}(s)$.  It thus
  remains to analyse the sums over the pairs of coincidence
  reflections in Proposition~\ref{prop:rational-gen}.  By
  Lemma~\ref{lem:BRS}, summing over all pairs of coincidence
  reflections is equivalent to summing over all four-element subsets
  $\{\pm w, \pm z\}$ of primitive orthogonal lattice vectors. Since
  these sets are disjoint, we can as well sum over all primitive
  vectors in $\Gam$, obtaining each summand exactly four times.  As
  earlier, we denote by $\Gam_w$ the BRS-sublattice corresponding to
  $\{\pm w, \pm z\}$, and we define $\sigma(w) := [\Gam : \Gam_w ]$,
  the index of $\Gam_w$ in $\Gam$. Finally, we use the notation
  $\kappa(w)=\frac{|w|}{|z|}$ for the quantity $\kappa$ introduced in
  Remark~\ref{rem:kappa}.  We thus obtain
\begin{align*}
  \Phi^{}_{\Gam, \mathsf{wr}}(s) - 2\ts \Phi^{\triangle}_{\Gam}(s) & \, =
  \, \frac{1}{4} \sum_{\substack{ w \text{ primitive}\\
      \sigma(w) \text{ odd}}}
  \frac{1}{\sigma(w)^s}\, \phi^{}_{\mathsf{wr},\text{even}}(\kappa(w);s) \\
  & \quad + \frac{1}{4} \sum_{\substack{ w \text{ primitive}\\
      \sigma(w) \text{ even}}}
  \frac{1}{\sigma(w)^s}\ts \left(\phi^{}_{\mathsf{wr},\text{even}}
    (\kappa(w);s) + 2^s \ts 
       \phi^{}_{\mathsf{wr},\text{odd}}(\kappa(w);s)\right) , 
\end{align*}
where the factor $\frac{1}{4}$ reflects the four elements of
$\{\pm w, \pm z\}$, as observed above.

From now on, we assume w.l.o.g.\ that $\Gam$ is integral and
primitive. Then, by Proposition~\ref{prop:index-of-BRS} of
Appendix~\ref{app-sec:index}, we have $\sigma(w)=
\frac{(w,w)}{g^*(w)}$, and $\kappa(w)= \frac{g^*(w)}{\sqrt{d}}$, where
$d$ is the discriminant of $\Gam$ and $g^*(w)$ is the coefficient of
$w$ in $\Gam^*$.  By Proposition~\ref{prop:index-of-BRS}, $g^*(w)$ is
a divisor of $d$, and can therefore take only a finite number of
distinct values. As a consequence, also $\kappa (w)$ takes only
finitely many values.  Moreover, $g^* (w)$ and $\kappa(w)$ are constant on
the cosets of an appropriate sublattice of $\Gam$. Accordingly, we can
subdivide the above summation into finitely many sums of simpler type.

To work this out explicitly, we choose a basis $\{v_1,v_2\}$ of
$\Gam^*$ such that $\{ v_1, d v_2 \}$ is a basis of $\Gam$, as in
Appendix~\ref{app-sec:index}. Using the quadratic form $Q(m,n) :=
|mv_1+ndv_2|^2$, and similarly the notation $g^*(m,n) :=
g^*(mv_1+ndv_2)$, $\sigma(m,n) := \sigma(mv_1+ndv_2)$ and $\kappa(m,n)
:= \kappa(mv_1+ndv_2)$, for $(m,n)\in \ZZ^2$, we have
$g^*(m,n)=\gcd(m,d)$ and $\sigma(m,n)=\frac{Q(m,n)}{g^*(m,n)}$, by
formula~\eqref{eq:g-star}, assuming $\gcd(m,n)=1$.  It follows from
Proposition~\ref{prop:parity-of-index} that the parity of
$\sigma(m,n)$ only depends on $\gcd(m,D)$ and $\gcd(n,2)$, where
$D=\lcm(2,d)$, and if the residues $m \bmod D$ and $n \bmod 2 $ are
fixed, the index $\sigma(m,n)$ only depends on $Q(m,n)$.  Hence,
\begin{align*}
  \Phi^{}_{\Gam, \mathsf{wr}}(s) - 2\ts \Phi^{\triangle}_{\Gam}(s) \, = \, &
  \frac{1}{4} \sum_{\gcd(m,n)=1}
  \frac{\gcd(m,d)^s}{Q(m,n)^s}  \\
  &\quad \times \left(\phi^{}_{\mathsf{wr},\text{\rm even}}(\kappa(m,n);s)
    + \delta_\sigma(m,n) \, 2^s \ts \phi^{}_{\mathsf{wr},\text{\rm
        odd}}(\kappa(m,n);s)\right)
  \nonumber \\
  \, = \, & \frac{1}{4} \sum_{k\mid D} \sum_{\ell \mid 2}
  \left(\phi^{}_{\mathsf{wr},\text{\rm even}}(\kappa(k,\ell);s) +
    \delta_\sigma(k,\ell)\, 2^s \ts 
  \phi^{}_{\mathsf{wr},\text{\rm odd}}(\kappa(k,\ell);s)\right)  \\
  &\quad \times \sum_{\substack{ \gcd(m,n)=1 \\ \gcd(m,D)=k \\
      \gcd(n,2)=\ell}} \frac{\gcd(k,d)^s}{Q(m,n)^s} \, , \nonumber
\end{align*}
where $\delta_\sigma$ is defined by
\[
\delta_\sigma(m,n) \, := \,
   \begin{cases}
       1 & \text{ if $\sigma(m,n)$ is even} \\
       0 & \text{ if $\sigma(m,n)$ is odd}
   \end{cases}
\]
and depends on $\gcd(m,D)$ and $\gcd(n,2)$, only. By
Remark~\ref{rem:phi-irr}, both $\phi^{}_{\mathsf{wr},\text{\rm
    even}}(\kappa(k,\ell);s)$ and $\phi^{}_{\mathsf{wr},\text{\rm
    odd}}(\kappa(k,\ell);s)$ are analytic in the open half plane
$\{\RE(s)> \frac{1}{2}\}$ except for $s=1$, where both have a simple
pole. Invoking Appendix \ref{app-sec:epst}, this is true of
\[
   \sum_{\substack{ \gcd(m,n)=1 \\ \gcd(m,D)=k \\ \gcd(n,2)=\ell}}
                   \frac{1}{Q(m,n)^s}
\]
as well, which shows that $\Phi^{}_{\Gam, \mathsf{wr}}(s) -
2\ts \Phi^{\triangle}_{\Gam}(s)$, and thus $\Phi^{}_{\Gam,
  \mathsf{wr}}(s)$, has a pole of order $2$ at $s=1$ and is analytic
elsewhere in $\{\RE(s)> \frac{1}{2}\}$, as claimed. The asymptotic
behaviour now follows from an application of Delange's theorem;
compare Theorem~\ref{thm:meanvalues}.
\end{proof}

At this stage, it remains an open question whether, in the general
rational case, the growth rate behaves as $c_1 x\log(x) + c_2 x +
\oo(x)$, like for the square and hexagonal lattices.

\appendix

\section{Some useful results from analytic
number theory}\label{app-sec:ana}

In what follows, we summarise some results from analytic number
theory that we need to determine certain asymptotic properties of the
coefficients of Dirichlet series generating functions. {}For the general
background, we refer to \cite{Apostol} and \cite{Zagier}.

Consider a Dirichlet series of the form $F(s)=\sum_{m=1}^{\infty} a(m)
m^{-s}$. We are interested in the summatory function $A(x)=\sum_{m\leq
  x} a(m)$ and its behaviour for large $x$. Let us give one classic
result for the case that $a(m)$ is real and non-negative.
\begin{theorem} \label{thm:meanvalues} Let $F(s)$ be a Dirichlet
  series with non-negative coefficients which converges for $\RE (s) >
  \alpha > 0$. Suppose that $F(s)$ is holomorphic at all points of the
  line $\{ \RE (s) = \alpha \}$ except at $s=\alpha$. Here, when
  approaching $\alpha$ from the half-plane to the right of it, we
  assume $F(s)$ to have a singularity of the form $F(s) = g(s) +
  h(s)/(s-\alpha)^{n+1}$ where $n$ is a non-negative integer, and both
  $g(s)$ and $h(s)$ are holomorphic at $s=\alpha$. Then, as
  $x\rightarrow\infty$, we have
\begin{equation} \label{meanvalues}
     A(x) \; := \; \sum_{m\leq x} a(m)
          \; \sim \;  \frac{h(\alpha)}{\alpha\cdot n!}
             \; x_{}^{\alpha} \, \bigl(\log(x)\bigr)^{n}  .
\end{equation}
\end{theorem}
The proof follows easily from Delange's theorem, for instance by
taking $q=0$ and $\omega=n$ in Tenenbaum's formulation of it; see
\cite[ch.\ II.7, Thm.\ 15]{Tenenbaum} and references given there.

The critical assumption in Theorem~\ref{thm:meanvalues} is the
behaviour of $F(s)$ along the line $\{ \RE (s) = \alpha \}$. In all
cases where we apply it, this can be checked explicitly. To
do so, we have to recall a few properties of the Riemann zeta
function $\zeta(s)$, and of the Dedekind zeta functions of
imaginary quadratic fields.

It is well-known that $\zeta(s)$ is a meromorphic function in the
complex plane, and that it has a sole simple pole at $s=1$ with
residue $1$; see \cite[Thm.\ 12.5(a)]{Apostol}.  It has no zeros in
the half-plane $\{ \RE (s)\geq 1 \}$; compare \cite[ch.\ II.3, Thm.\
9]{Tenenbaum}.  The values of $\zeta(s)$ at positive even integers are
known \cite[Thm.\ 12.17]{Apostol} and we have
\begin{equation}
     \zeta(2) \; = \; \frac{\pi^2}{6} \ts .
\end{equation}
This is all we need to know for this case.

Let us now consider an imaginary quadratic field $K$, written
as $K=\QQ(\sqrt{d}\,)$ with $d<0$ squarefree. The corresponding
discriminant is 
\[
   D \, = \, \begin{cases}
   4 \ts d, & \text{if $d\equiv 2,3 \bmod{4}$}, \\
    d ,    & \text{if $d\equiv 1 \bmod{4}$},
   \end{cases}
\]
see \cite[\S 10]{Zagier} for more. We need the Dedekind zeta function
of $K$ (with fundamental discriminant $D < 0$). It follows from
\cite[\S 11, Eq.\ (10)]{Zagier} that it can be written as
\begin{equation}
  \zeta^{}_{K}(s) \; = \; \zeta(s) \cdot L(s,\chi^{}_{D}) 
\end{equation}
where $L(s,\chi^{}_{D})=\sum_{m=1}^{\infty}\chi^{}_{D}(m)\, m^{-s}$ is
the $L$-series \cite[Ch.\ 6.8]{Apostol} of the primitive Dirichlet
character $\chi^{}_{D}$. The latter is a totally multiplicative
arithmetic function, and thus completely specified by
\begin{equation}
  \chi^{}_{D}(p) \, = \, \bigl( \myfrac{D}{p} \bigr), 
\end{equation} 
for odd primes, where $\bigl( \frac{D}{p} \bigr)$ is the usual
Legendre symbol, together with
\[
    \bigl( \myfrac{D}{2} \bigr) \, = \, \begin{cases}
     0, & \text{if $D\equiv 0 \bmod{4}$}, \\
     1, & \text{if $D\equiv 1 \bmod{8}$}, \\
     -1, & \text{if $D\equiv 5 \bmod{8}$}.  \end{cases}
\]
$L(s,\chi^{}_{D})$ is an entire function \cite[Thm.\ 12.5]{Apostol}.
Consequently, $\zeta^{}_{K}(s)$ is meromorphic, and its only pole is
simple and located at $s=1$. The residue is $L(1,\chi^{}_{D})$, and
from \cite[\S 9, Thm.~2]{Zagier} we get the simple formula
\begin{equation}\label{Lseries-at-one}
  L(1, \chi^{}_{D}) \, = \, - \ts \myfrac{\pi}{\lvert D \rvert^{3/2}}
  \sum_{n=1}^{\lvert D \rvert -1} n \, \chi^{}_{D} (n) \ts .
\end{equation}
In particular, for the two fields $\QQ(\ii)$ and $\QQ(\rho)$,
one has the values $\pi/4$ and $\pi/ 3 \sqrt{3}$, respectively.
\medskip

Our next goal is an estimate on sums of the form
$\sum_{\ell < n < \alpha \ell} n^{-s}$ for $\ell \in \NN$,
$\alpha > 1$ and $s>0$. Invoking Euler's summation formula from
\cite[Thm.~3.1]{Apostol}, one has
\begin{equation}\label{eq:euler-sum}
   \sum_{\ell < n \le \alpha \ell} \myfrac{1}{n^{s}} \; = \,
   \int_{\ell}^{\alpha \ell} \myfrac{\dd x}{x^{s}} \, -
   \int_{\ell}^{\alpha \ell} \bigl(x - [x] \bigr)\,
   \myfrac{ s \dd x}{x^{s+1}} \, + \, 
   \myfrac{[\alpha \ell] - \alpha \ell}{(\alpha \ell)^{s}}
   \, - \, \myfrac{[\ell] - \ell}{\ell^{s}} \ts .
\end{equation}
The last term vanishes (since $\ell\in\NN$), while the second last
does whenever $\alpha\ell \in \NN$ (otherwise, it is negative).
Since the second integral on the right hand side is strictly
positive (due to $\alpha > 1$), we see that
\begin{equation}\label{eq:int-def}
    \sum_{\ell < n < \alpha \ell} \myfrac{1}{n^{s}} \; \le 
    \sum_{\ell < n \le \alpha \ell} \myfrac{1}{n^{s}} \, < \,
    I_{s} \, :=   \int_{\ell}^{\alpha \ell} \myfrac{\dd x}{x^{s}} 
    \, = \, \myfrac{1-\alpha^{1-s}}{s-1} \ts \ell^{1-s} .
\end{equation}
Observing next (once again due to $\alpha > 1$) that
\[
     \int_{\ell}^{\alpha \ell} \bigl(x - [x] \bigr)\,
     \myfrac{ s \dd x}{x^{s+1}} \, < \, \myfrac{1}{\ell^{s}}
     \, - \, \myfrac{1}{(\alpha \ell)^{s}} \ts ,
\]
one can separately consider the two cases $\alpha \ell \not\in \NN$ 
and $\alpha \ell \in \NN$ to verify that we always get
\[
   \sum_{\ell < n < \alpha \ell} \myfrac{1}{n^{s}} \, >
   I_{s} - \myfrac{1}{\ell^{s}} \ts .
\]
This can immediately be generalised to sums of the form $\sum_{\ell <
  n < \alpha \ell + \beta} (n+\gamma)^{-s}$ with $\beta,\gamma \ge 0$,
which we summarise as follows.

\begin{lemma}\label{lem:est-sum}
  Let\/ $\ell\in\NN$, $\alpha > 1$, $\beta \ge 0$ and\/ $0\le \gamma <
  1$.  If $s \ge 0$, one has the estimate
\[
   I_{s} - \myfrac{1}{(\ell + \gamma)^{s}} \;  <
    \sum_{\ell < n < \alpha \ell + \beta} \myfrac{1}{(n + \gamma)^{s}}
    \, < \, I_{s} \ts ,
\]
  with the integral $I_{s} = \int_{\ell}^{\alpha \ell + \beta}
  \frac{\dd x}{(x+\gamma)^{s}} $ as the generalisation of
  that in Eq.~\eqref{eq:int-def}.  \qed
\end{lemma}

Let us finally mention that
\[
   \myfrac{1 - \alpha^{1-s}}{s-1}  \, = \,
   \log (\alpha) \sum_{m\ge 0} 
   \frac{\bigl( \log (\alpha) \ts (1-s)\bigr)^{m}}{(m+1) !} \ts ,
\]
so that this function is analytic in the entire complex plane. 
In particular, one has the asymptotic expression
$\frac{1 - \alpha^{1-s}}{s-1} = \log (\alpha) + \mathcal{O} \bigl(
\lvert 1-s \rvert \bigr)$ for $s \to 1$.

\section{Asymptotics of similar sublattices}\label{app-sec:asym}

We have sketched how to determine the asymptotics of the number of
well-rounded sublattices of the square and hexagonal lattices. As a
by-product of these calculations, and as a refinement of the results
from~\cite{Baake+Grimm}, we obtain the asymptotics of the number of
similar and primitive similar sublattices as follows.

\begin{theorem} 
  The asymptotics of the number of similar and of primitive similar
  sublattices of the square lattice is given by
\begin{align}
\label{eq:asym-ssl-sq}
  \sum_{n\leq x}b^{}_{\square}(n) 
     & \, = \, L(1,\chi^{}_{-4})\,x + \OO\bigl(\sqrt{x}\, \bigr) 
        \, = \, \frac{\pi}{4}\,x + \OO\bigl(\sqrt{x}\, \bigr) \\
\shortintertext{and}
  \sum_{n\leq x}b^{\mathsf{pr}}_{\square}(n)  
       & \, = \, \frac{L(1,\chi^{}_{-4})}{\zeta(2)}\, x + 
          \OO\bigl(\sqrt{x}\log(x)\bigr) \, = \, 
          \frac{3}{2\pi}\,x + \OO\bigl(\sqrt{x}\log(x)\bigr) \ts .
\end{align}
\end{theorem}

\begin{proof}[Sketch of proof]
  Note that $b^{}_{\square}(n)=(\chi^{}_{-4}*1)(n)$. We now get the
  asymptotics of its summatory function by an application of
  Eq.~\eqref{eq:diri-conv2}.  Observe
  $b^{\mathsf{pr}}_{\square}=\nu*b^{}_{\square}$, where
  $\nu(n):=\mu(\sqrt{n})$ is defined to be $0$ if $n$ is not a square
  and $\mu$ is the Moebius function.  An application of
  Eq.~\eqref{eq:diri-conv1} then yields the result.
\end{proof}

Similarly, one proves the following result.
\begin{theorem} 
  The asymptotics of the number of similar and of primitive similar
  sublattices of the hexagonal lattice is given by
\begin{align}
\label{eq:asym-ssl-tr}
  \sum_{n\leq x}b^{}_{\triangle}(n) 
     & \, = \, L(1,\chi^{}_{-3})\,x + \OO\bigl(\sqrt{x}\, \bigr) 
        \, = \, \frac{\pi}{3\sqrt{3}}\,x + \OO\bigl(\sqrt{x}\, \bigr) \\
\shortintertext{and}
  \sum_{n\leq x}b^{\mathsf{pr}}_{\triangle}(n)  
       & = \frac{L(1,\chi^{}_{-3})}{\zeta(2)}\, x + 
            \OO\bigl(\sqrt{x}\log(x)\bigr)
           = \frac{2}{\pi\sqrt{3}}\,x + \OO\bigl(\sqrt{x}\log(x)\bigr)),
\end{align}
as $x\to\infty$.    \qed
\end{theorem}

\section{The index of BRS sublattices}\label{app-sec:index}

Let us complement the discussion of rational orthogonal frames and BRS
sublattices as introduced in Lemma~\ref{lem:BRS}.  We start with an
arbitrary rational, primitive, planar lattice $\Gam$ and denote by
$(v,w) \in \ZZ$ with $v,w \in \Gam$ the given positive definite
integer-valued primitive symmetric bilinear form on $\Gam$, extended
to the rational space $\QQ \Gam$.  Primitivity means that the form is
not a proper integral multiple of another form; it is equivalent to
the condition that $\gcd (a,b,c) = 1$, where $G =
\left( \begin{smallmatrix} a & b \\ b & c \end{smallmatrix} \right)$
is the Gram matrix with respect to an arbitrary basis $v_1,v_2$ of
$\Gam$.

In the following, we need the notion of the \emph{coefficient}
$g^{}_\Gam(v)$ of an arbitrary vector $v \in \QQ \Gam$ with respect to
$\Gam$. This is the unique positive rational number $g$ such that
$v=gv_0$, where $v_0 \in \Gam$ is primitive in $\Gam$. Equivalently,
$g^{}_\Gam(v)$ is the unique positive generator of the rank one
$\ZZ$-submodule of $\QQ$ consisting of all $q \in \QQ$ such that
$q^{-1}v \in \Gam$. So, a vector $v$ is primitive in $\Gam$ if and
only if $g^{}_\Gam (v)=1$, in accordance with the first
definition. Still another description of $g^{}_\Gam(v)$ is the gcd
(taken in $\QQ$) of the coefficients of $v$ with respect to an
arbitrary $\ZZ$-basis of $\Gam$. Below, we shall use the coefficient
$g^* := g^{}_{\Gam^*}$ in particular with respect to the dual lattice
$\Gam^* := \{w \in \QQ \Gam \mid \forall v \in \Gam : (v,w) \in \ZZ
\}$.

For an arbitrary primitive vector $w \in \Gam $, we recall the
notation $\Gam_w$ for the BRS sublattice spanned by $w$ and its
orthogonal sublattice $w^\perp \cap \Gam$, i.e.\ by $w$ and $z$, where
$z$ is the primitive lattice vector orthogonal to $w$ (unique up to
sign).  The main result of this appendix is to compute the index of
$\Gam_w \in \Gam$ as follows.

\begin{prop}\label{prop:index-of-BRS}
  Let $w$ be a primitive vector in a planar lattice $\Gam$ with
  primitive symmetric bilinear form, let $g^*(w)$ denote its
  coefficient in the dual lattice $\Gam^*\subseteq \Gam$. Then,
  $g^*(w)$ is a divisor of the discriminant $d$ of the lattice, and
\[
   [\Gam : \Gam_w ] \, = \, \frac{(w,w)}{g^*(w)} \, .
\]
\end{prop}

\begin{proof}
  The first claim follows easily from the fact that $d$ is equal to
  the order of the factor group $\Gam^* / \Gam$, but it is also a
  consequence of the following computation leading to a proof of the
  second claim.  Since $w$ is primitive, we can complement it to a
  basis $v_1=w,v_2$ of $\Gam$. Consider the dual basis $v^*_1, v^*_2$
  with respect to the given scalar product; it is a $\ZZ$-basis of
  $\Gam^*$. Writing the above vector $z$ as $z=sv^*_1 +tv^*_2$ with
  $s,t \in \ZZ$ clearly leads to $s=0$, and $t$ is the smallest
  integer such that $tv^*_2 \in \Gam$. If $G$ is the Gram matrix with
  respect to $v_1,v_2$ as above, then $G$ is also the transformation
  matrix which expresses the original basis vectors $v_1,v_2$ in terms
  of their dual vectors, in particular $v_1=a v^*_1 + b v^*_2$, which
  shows that the coefficient of $w=v_1$ in $\Gam^*$ is
\[
   g^*(w) \, = \, \gcd(a,b) \ts .
\]
On the other hand, with $d:= ac-b^2$,
\[
  G^{-1} \, = \, \frac 1 d \left(\begin{matrix}
      c & -b\\-b & a \end{matrix} \right)
\]
is the transformation matrix expressing the dual basis in terms of the
original basis. In particular
\[
v^*_2 \, = \, \frac 1 d (-b v_1 + a v_2) \ts ,
\]
which implies that
\[
   t \, = \, \frac{d}{\gcd(a,b)} \ts .
\]
To compute the index of $\Gam_w$ in $\Gam$, we use the bases $v_1,v_2$
of $\Gam$ and $v_1,tv^*_2$ of $\Gam_w$. The corresponding
transformation matrix is $ \left(\begin{smallmatrix} 1 &
    -\frac{b}{d}t\\0 & \frac a d t \end{smallmatrix} \right) $, which
has determinant
\[
   \frac{a}{ d} \,  t \, = \,
   \frac {a} {d}\, \frac{d}{\gcd(a,b)} 
    \, = \, \frac{a}{g^*(w)} \ts ,
\]
as claimed.
\end{proof}

Since the vector $w$ was assumed primitive in $\Gam$, it is even true
that $g^*(w)$ is a divisor of the exponent of the factor group $\Gam^*
/ \Gam$. But from the primitivity of the bilinear form it follows that
this factor group is actually cyclic of order $d$, so its exponent is
equal to $d$, and we do not get an improvement: all divisors of the
discriminant $d$ can occur as a value $g^*(w)$.

It is easy to see that the quantity $g^*(w)$ only depends on an
appropriate coset of $w$; in fact, under the assumptions of the last
proposition, the coset modulo $d\Gam^*$ suffices. For purposes of
reference, we state this as an explicit remark.

\begin{remark}\label{rem:g-star}
  Under the assumptions of Proposition~\ref{prop:index-of-BRS}, let
  $w,w'$ be primitive such that\/ $w \equiv w' \pmod{d\Gam^*}$. Then,
  $g^*(w)=g^*(w')$.
\end{remark}

For explicit computations involving $g^*$, it is convenient to use a
basis corresponding to the elementary divisors of $\Gam$ in $\Gam^*$,
that is, a basis $\{v_1,v_2\}$ of $\Gam^*$ such that $\{ v_1, d v_2
\}$ is a basis of $\Gam$. The primitive vectors in $\Gam$ read $w =
mv_1 + ndv_2$ with $\gcd(m,n) = 1$. Using $g := \gcd(m,d)$, we can
rewrite this as $w=g \ts ( (m/g)v_1 + n(d/g)v_2 )$, where the coefficients
$m/g$ and $n(d/g)$ are coprime, in other words, $(m/g)v_1 + n(d/g)v_2$
is primitive in $\Gam^*$. This proves
\begin{equation}\label{eq:g-star}
   g^*(mv_1 + ndv_2) \, = \, \gcd(m,d) \ts ,
   \quad \text{ if } \gcd(m,n)=1 \ts .
\end{equation}
Notice that this formula again proves Remark~\ref{rem:g-star}.

For our application to well-rounded sublattices, we also have to
consider the parity of the index $[\Gam : \Gam_w ]$. For this, we need
the following refinement of Remark~\ref{rem:g-star}.

\begin{prop}\label{prop:parity-of-index}
  Under the assumptions of Proposition~$\ref{prop:index-of-BRS}$, let
  $w,w'$ be primitive such that $w \equiv w' \pmod{d\Gam^*}$ and $w
  \equiv w' \pmod{2 \Gam}$. Then,
  $[\Gam : \Gam_w ] \equiv [\Gam : \Gam_{w'} ] \pmod{2}$.
\end{prop}

\begin{proof}
  The proof is of course based on Proposition~\ref{prop:index-of-BRS},
  taking into account that, under our assumptions, $g:= g^*(w) =
  g^*(w')$, by Remark~\ref{rem:g-star}. First of all, recall that $g$
  divides $d$.  Now, we write $w' = w + u = w + du'$ with $u' \in
  \Gam^*$ and $u \in 2 \Gam$, and we compute explicitly
\[
  \frac{(w',w')}{g} \, = \, \frac{(w,w)}{g} + 2\ts \frac{d}{g}(w,u') +
  \frac{d}{g}(u,u') \, \equiv \, \frac{(w,w)}{g} \pmod{2} \ts .
\]
Notice that the last inner product $(u,u')$ is indeed in $2\ZZ$, since
$u \in 2\Gam$ and $u' \in \Gam^*$.
\end{proof}

\section{Epstein's $\zeta$-function}\label{app-sec:epst}

For a quadratic form $Q(m,n)= a m^2 + 2b mn + c n^2$, the Epstein
$\zeta$-function is defined as
\begin{equation}
   \zeta^{}_Q(s) \, :=  \sum_{(m,n)\neq (0,0)} \frac{1}{Q(m,n)^s} \ts ,
\end{equation}
where the sum runs over all non-zero vectors $(m,n)\in\ZZ^2$.  The
series converges in the half plane $\{\RE(s)>1\}$. It has an analytic
continuation which is a meromorphic function in the whole complex
plane with a single simple pole at $s=1$ with residue
$\frac{\pi}{\sqrt{d}}$, where $d=ac-b^2$ as before;
see~\cite{KK,Siegel}.  It is closely connected to
\begin{equation}
  \zeta^{\mathsf{pr}}_Q(s)\, := \sum_{(m,n)=1} \frac{1}{Q(m,n)^s}
  \, = \,  \frac{1}{\zeta(2s)}\, \zeta^{}_Q(s) \ts ,
\end{equation}
where the sum runs over all pairs of integers that are relatively prime.
In the explicit summations, we now use $(m,n)$ instead of $\gcd(m,n)$.

In Section~\ref{sub:rational}, we need the sum
\begin{equation}
  \sum_{ \substack{ (m,n)=1 \\ (m,D)=k \\ (n,C)=\ell}} \frac{1}{Q(m,n)^s} \, ,
\end{equation}
where $C,D,k,\ell$ are some fixed positive integers with $k,\ell$
relatively prime. Using the Moebius $\mu$-function, we can express
\begin{equation}
  \sum_{\substack{ (m,n)=1 \\ (m,D)=k \\ (n,C)=\ell}} \frac{1}{Q(m,n)^s}
  \; = \sum_{\substack{ (m,n)=1 \\ (m,\ell D/k)=1 \\ (n,kC/\ell)=1}} 
    \frac{1}{Q(k m,\ell n)^s}
  \,  = \sum_{c\mid \frac{ \ell D}{k}} \mu(c) \; 
  \varphi^{}_Q \left(c\frac{kC}{\ell};c k,\ell;s\right)
\end{equation}
in terms of
\begin{equation}
  \varphi^{}_Q(a;k,\ell;s)\, := \sum_{\substack{ (m,n)=1 \\ (n,a)=1}}
  \frac{1}{Q(k m,\ell n)^s}\, .
\end{equation}
As $Q(m,n)$ is homogeneous of degree $2$, we have
\begin{equation}
  \varphi^{}_Q(a;k b,\ell b;s) \, = \,
    \frac{1}{b^{2s}}\, \varphi^{}_Q(a;k,\ell;s).
\end{equation}
Furthermore, observe that $\varphi^{}_Q(a;k,\ell;s)=
\varphi^{}_Q(b;k,\ell;s)$, whenever $a$ and $b$ have the same prime
factors. In particular, we may assume that $a$ is squarefree in the
following.  Using the same methods as above, we can derive the
following recursion
\begin{equation}
  \varphi^{}_Q(a;k,\ell;s) \, = 
  \sum_{b \mid a} \sum_{c \mid \frac{a}{b}} \mu(c) \ts
   \frac{1}{b^{2s}}\ts \varphi^{}_Q(b;k,c \ell;s)\ts ,
\end{equation}
where we have made use of the assumption that $a$ is squarefree and
employed the multiplicativity of $\mu$. This recursion has the
solution
\begin{equation}
   \varphi^{}_Q(a;k,\ell;s) \, = \, 
   \left( \prod_{p \mid a} \frac{1}{1-p^{-2s}} \right)
   \left( \sum_{b \mid a} \mu(b) \, \varphi^{}_Q(1;k,b \ell;s) \right),
\end{equation}
where the product is taken over all primes $p$ dividing $a$.  As
$\varphi^{}_Q(1;k,b\ell;s)$ is the primitive Epstein $\zeta$-function
$\zeta^{\mathsf{pr}}_{\tilde Q}(s)$ corresponding to the quadratic
form $\tilde Q(m,n)= Q( km, b\ell n)$, this shows that
$\varphi^{}_Q(a;k,\ell;s)$ and thus
\[
   \sum_{\substack{(m,n)=1 \\ (m,D)=k \\ (n,C)=\ell}} \frac{1}{Q(m,n)^s} 
\]
are sums of Epstein zeta functions, and thus are meromorphic functions
with a simple pole at $s=1$ and analytic elsewhere in
$\{\RE(s)>\frac{1}{2}\}$.

Alternatively, we can obtain this result by an application of
Theorem~3 in~\cite{Siegel}; see also \cite{KS}. Applied to our
situation, it states that
\begin{equation}
   \psi^{}_Q (D,C,i,j;s) \; :=
   \sum_{\substack{ m \equiv i (D) \\ n\equiv j (C)}} \frac{1}{Q(m,n)^s} 
\end{equation}
has an analytic continuation, which is analytic in the entire complex
plane except for a simple pole at $s=1$ with residue
$\frac{\pi}{\sqrt{\det (Q')}}$, where $Q'(m,n):= Q(Dm,Cn)$.  Using
methods similar to those in~\cite{BMP-visible,PleaHuck}, we first
observe for $k,\ell$ coprime
\begin{align*}
\sum_{\substack{ (m,n)=1 \\ (m,D)=k \\ (n,C)=\ell}} \frac{1}{Q(m,n)^s}\, 
& = \sum_{\substack{ (m,D)=k \\ (n,C)=\ell}} 
         \frac{1}{Q(m,n)^s} \sum_{r\mid (m,n)} \mu (r) \\
& = \sum_{r\in\NN } \mu(r) \frac{1}{r^{2s}} 
      \sum_{\substack{ (rm,D)=k \\ (rn,C)=\ell}} 
           \frac{1}{Q(m,n)^s}  \\[1mm]
& =  \sum_{u\mid k} \sum_{v \mid \ell} \sum_{\substack{ r\in\NN \\ (r,CD)=1}}
         \frac{\mu(uvr)}{(uvr)^{2s}} \sum_{\substack{ (uvrm,D)=k \\ 
         (uvrn,C)=\ell}} \frac{1}{Q(m,n)^s} \, .  
\end{align*}
As $r$ is coprime with $C$ and $D$ we see that
\begin{equation}
  \sum_{\substack{ (uvrm,D)=k \\ (uvrn,C)=\ell}} 
           \frac{1}{Q(m,n)^s}
  \; = \sum_{\substack{ (vm,D/u)=k/u \\ (un,C/v)=\ell/v}} 
           \frac{1}{Q(m,n)^s}
\end{equation}
is independent of $r$. Moreover, the latter sum can be written as a
(finite) sum of suitable functions of the form $\psi^{}_Q(D,C,i,j;s)$
and therefore it is analytic in the entire complex plane except for a
simple pole at $s=1$. As $u,v,r$ are coprime, $\mu(uvr)=
\mu(u)\mu(v)\mu(r)$, and hence the only remaining infinite sum
\begin{equation}
  \sum_{\substack{ r\in\NN \\ (r,CD)=1}} \frac{\mu(r)}{r^{2s}}
  \, = \, \frac{1}{\zeta(2s)} \prod_{p\mid CD} \frac{1}{1-p^{2s}}
\end{equation}
is analytic in $\{\RE(s)>\frac{1}{2}\}$, which again shows that
\[
  \sum_{ \substack{ (m,n)=1 \\ (m,D)=k \\ (n,C)=\ell}} \frac{1}{Q(m,n)^s} 
\]
is a meromorphic function with a simple pole at $s=1$ and analytic
elsewhere in $\{\RE(s)>\frac{1}{2}\}$.

\section*{Acknowledgements}

The authors thank S.~Akiyama, J.~Br{\"u}\-dern and L.~Fukshansky for
discussions, R.~Schulze-Pillot for bringing Siegel's work on Epstein's
zeta function to our attention, and S.~K{\"u}hn\-lein for
making~\cite{Kuehn} available to us prior to publication. This work
was supported by the German Research Foundation (DFG), within the CRC
701.

\end{document}